\documentclass{article}
\usepackage[margin=2.5cm]{geometry}
\usepackage{color}
\usepackage{amsfonts,amsmath,amsthm}

\usepackage{algorithmic}
\usepackage{algorithm}
\usepackage[numbers]{natbib}

\usepackage{graphicx}
\usepackage{pgf,tikz}
\usetikzlibrary{arrows}

\newtheorem{lemma}{Lemma} 
\newtheorem{theorem}{Theorem} 
\newtheorem{definition}{Definition} 
\newtheorem{corollary}{Corollary} 
\newtheorem{example}{Example}

\newcommand{\opt}{\mathrm{opt}}
\newcommand{\non}{\mathrm{non}}

\begin{document}
\title{A Unified Markov Chain Approach to Analysing Randomised Search Heuristics }
 
\author{Jun He \thanks{Department of Computer Science, Aberystwyth University, Aberystwyth, SY23 3DB, U.K.} 
\and Feidun He  
\and Xin Yao   } 
 
\date{\today}

\maketitle
\begin{abstract}
The convergence,   convergence rate  and  expected hitting time  play fundamental roles in the  analysis of randomised search heuristics. This paper presents a unified   Markov chain approach to studying  them.   Using the approach,  the sufficient and necessary conditions of convergence in distribution are established. Then the average convergence rate is introduced to randomised search heuristics and its lower and upper bounds are derived. Finally, novel average  drift analysis and backward drift analysis are proposed for bounding the   expected hitting time. A computational study is also conducted to investigate the convergence,   convergence rate  and  expected hitting time. The theoretical study belongs to a prior and general study while the computational  study belongs to a posterior and case study.
\end{abstract}

\section{Introduction}
Randomised search heuristics, such as evolutionary algorithms,   have been  widely applied to optimization problems. Randomised search heuristics  belong to   iterative methods.  As iterative methods,   the following three  questions are fundamental in both theory and practice. 
\begin{enumerate}
\item   (\emph{Convergence}) whether is a randomised search heuristic able to find an optimal solution eventually?
\item  (\emph{Convergence rate}) how fast does a randomised search heuristic converge to the optimal set per iteration?  

\item  (\emph{Hitting time}) how many iterations are needed for obtaining an optimal solution?

\end{enumerate}

Most randomised search heuristics   satisfy the Markov property, that is,  a   population sequence (where a population consists of one or more solutions) is generated subject to some probability distribution; and  the state of current population decides the state of next population in a probabilistic way. 
Hence Markov chain theory provides a theoretical framework for analysing   randomised search heuristics~\cite{rudolph1998finite,he2003towards},

Based on absorbing Markov chain theory~\cite{kemeny1960finite,grinstead1997introduction},   a unified approach is used for studying the convergence, convergence rate and   expected hitting time of randomised search heuristics in this paper. The idea  is described as follows: the population sequence generated by a randomised search heuristic is modelled by an absorbing Markov chain. Consider the probability distribution of a population in the   non-optimal solution set and represent it by a vector. Then the randomised search    is  equivalent to  a  matrix iteration. The vector 1-norm is chosen to measure the distance between a population and the  optimal solution set. This  new feature  makes our current analysis  different from previous work \cite{rudolph1994convergence,suzuki1995markov,he1999convergence}.  Using matrix iteration analysis~\cite{varga2009matrix,meyer2000matrix},  all theoretical results can be established in a unified manner.

The purpose of this paper is to seek new theoretical tools for analysing randomised search heuristics. Indeed we have developed three new tools in the paper, which are  novel average  drift analysis  and novel backward drift analysis for bounding the   expected hitting time;   and the average convergence rate  of randomised search heuristics.  These new tools are seldom studied before.

This paper is organised as follows: literature  review  is given in Section \ref{secReview}. The   Markov chain model appears in Section~\ref{secMakovChains}.   Convergence is analysed in Section~\ref{secConvergence}.  The average convergence rate of EAs is discussed in Section~\ref{secConvergenceRate}. The expected hitting time is analysed in Section~\ref{secHittingTimes}.  Final conclusions are described in Section~\ref{secConclusion}.

\section{Literature Review}
\label{secReview}
In this section we  review the work related to Markov chain analysis for randomised search heuristics,  and also show the difference between our work and previous ones. 

Markov chains have been applied into analysing  randomised search heuristics  more than two decades~\cite{goldberg1987finite,eiben1991global}. \cite{rudolph1998finite} gave a survey of  Markov chain analysis of evolutionary algorithms  up to 1998. \cite{oliveto2007time}   reviewed some achievements  after that year.  

Markov chain theory provides a theoretical framework to model randomised search heuristics. For example, \cite{nix1992modeling,davis1993markov,dejong1995using} modelled genetic algorithms by Markov chains. \cite{cantu2000markov} presented Markov chain models of parallel genetic algorithms.  In general, randomised search heuristics for discrete optimisation may be modelled by Markov chains, and randomised search heuristics for continuous optimisation by Markov processes. 

Markov chain theory is widely applied to the    limit behaviour of randomised search heuristics and their convergence. For example,   \cite{fogel1994asymptotic,rudolph1994convergence,he2001conditions}   proposed  different convergent conditions of genetic algorithms.  \cite{rudolph2000convergence} analysed convergence properties of some multi-objective evolutionary algorithms. \cite{chakraborty1996analysis} compared  various selection algorithms using Markov chains.
 
The   convergence rate of randomised search heuristics is a less studied topic. \cite{suzuki1995markov} analysed  a simple GA by evaluating the eigenvalues of the transition matrix of the Markov chain and computed its convergence rate.   \cite{he1999convergence} gave the convergence rates of general genetic algorithms by using the minorization
condition.  \cite{du2010convergence} studied the convergence rates of gene expression programming    by means of Markov chain and spectrum analysis. 

The   expected hitting time study  has received more  attentions recently.  \cite{he2001drift,droste2002analysis} made two initial discussions on  evolutionary algorithms. According to  absorbing Markov chain theory,  the  expected hitting time can be calculated based on the fundamental matrix. Using this approach, \cite{he2002individual}  compared  the expected hitting time of (1 + 1)   and (N + N) evolutionary algorithms.  \cite{he2003towards} gave a framework for analysing  the expected hitting time  of evolutionary algorithms. \cite{zhou2009comparative}   compared the runtime of three simple heuristic algorithms.  

However, the fundamental matrix approach can only be suitable for simple algorithms and problems. Thus  \cite{he1998study,he2001drift} introduced drift analysis  to the expected hitting time study. Currently drift analysis becomes a popular theoretical tool. Different variants have been developed, such as simplified drift analysis \cite{oliveto2011simplified}, multiplicative drift analysis   \cite{doerr2010multiplicative},   adaptive drift analysis     \cite{doerr2010adaptive},  and variable drift  \cite{mitavskiy2009theoretical,johannsen2010random}.  Drift analysis have been applied to  both  (1+1) EAs \cite{witt2012optimizing} and population-based EAs \cite{lehre2010negative}. Drift theorems can be established using either Markov chain theory or super-martingale theory~\citep{he2001drift,neumann2009analysis}.

So far the convergence,    convergence rate  and expected hitting time  are   studied separately. Different from existing work,  we present  a unified  approach  to bringing these three issues together.  

\section{Absorbing Makrov Chain Model}
\label{secMakovChains}
\subsection{Theoretical Study}
In this subsection, we describe absorbing Markov chain models of randomised search heuristics. 

Consider a maximisation problem on a finite state   space, that is 
\begin{equation}
   \max\,  f(x),   \qquad x \in D,
\end{equation}
where $x$ is a variable and $D$ is its definition domain, a finite set. $f(x)$ is called a \emph{fitness function}. 
 
A randomised search heuristic   can be viewed as a  randomised  iteration process: initially construct a population of solutions $\Phi_0$;  based on   $\Phi_0$, then probabilistically generate a new population of solutions $\Phi_1$;  based on $\Phi_1$, then  probabilistically generate a new population of solutions $\Phi_2$, and so on. This procedure is repeated until a stopping criterion is satisfied. Then a sequence of populations is produced
$$\Phi_0 \to \Phi_1 \to \Phi_2 \to \cdots.$$  

In order to guarantee that the best solution is always kept during the iteration, an extra archive  is added for recording the best found solution.   The archive itself is not involved in  generating   new  solutions. This strategy  is   called \emph{elitist}.
A  randomised search heuristic with an archive  is described in Algorithm \ref{alg1}.

\begin{algorithm}
\caption{randomised search heuristic with an archive} \label{alg1}
\begin{algorithmic}[1]
\STATE set counter  $t$ to 0; 
\STATE initialize a population of solutions $ \Phi_{0}$;
\STATE archive $\Lambda_0$  keeps the best solution in $\Phi_0$; 
\FOR{$t=0,1,2, \cdots$}
\STATE a new population of solutions $\Phi_{t+1}$ is generated from $\Phi_t$;
\STATE  update the archive $\Lambda_{t+1}$ if the best solution in $\Phi_{t+1}$ is better than $\Lambda_t$; 
\STATE  counter  $t$ is increased by $1$;
\ENDFOR
\end{algorithmic}
\end{algorithm}

A \emph{population} consists of one or more solutions while a solution is called an  \emph{individual}. The procedure of generating new a population may include several steps, such as  mutation, crossover and selection in a genetic algorithm.  For convenience of analysis, we only consider   non-dynamical algorithms whose search operators are not changed during the iteration. 
The algorithm  runs for ever. This assumption is taken for  convenience of   analysis of the hitting time.     
The fitness of a  population at the $t$-th iteration is defined by  the archive at the $t$-th iteration, denoted by $f_t$.

The sequence $\{ \Phi_t; t=0,1, \cdots \}$ can be formulated by a Markov chain. Let $X$ and $Y$ be two populations. The transition from $X$ to $Y$ happens with a probability $P(X,Y)$, that is,  
\begin{align}
P(X,Y):=P(\Phi_{t+1}= Y \mid  \Phi_t = X), \quad  X , Y \in S,
\end{align}
where $S$ denotes the set of all populations. Both $\Phi_t$ and $X$ represent a population, but   $\Phi_t$ is a random variable for representing the population at the $t$-th iteration; $X$  its value, or called a \emph{state} in the population space.  

For the sake of argument, we introduce an auxiliary Markov chain $\{\Phi'_t; t=0,1, \cdots\}$ as follow. Let $\Phi'_{t}  =\Phi_t$ before the optimal solution is found for the first time.  Once   an optimal solution is found at the $t$-th iteration, then assign
$
  \Phi'_{s}  =\Phi_t$ for any  iteration $s$ after $t$.
This implies the optimal set is always absorbing in the new chain $\{\Phi'_t; t=0,1, \cdots\}$. We don't care about the behaviour of randomised search heuristics after the first time to hit a optimal solution. To simplify notation, we still denote the new chain by $\{\Phi_t; t=0,1, \cdots\}$.

As a result, the population sequence $\{ \Phi_t; t=0,1, \cdots \}$   is modelled by  a homogeneous  Markov chain where the optimal solution set is always absorbing.

\subsection{Case Studies}
In this subsection, we show how randomised heuristics can be easily modelled by Markov chains. For the sake of illustration, we  consider  a simple maximisation problem  
\begin{align}
\label{equFunction}
&\max \, f(x), &x \in \{ 0,1, \cdots, 100 \}
\end{align}
 
 Two randomised search heuristics are applied to the above problem.
The first   algorithm  adopts random walk with elitist selection, denoted by RSH-I (see Algorithm \ref{alg2}).

\begin{algorithm}[ht]
\caption{RSH-I}
\label{alg2}  
\begin{algorithmic} 
\STATE   \textbf{Random Walk}: provided that $\Phi_t=x$, then  it walks  to $x-1$ (if $x-1$ is not less than $0$) with   probability $0.01$, or walks   to $x+1$ (if $x+1$ is not more than $100$)  with  probability $0.01$. Denote the new position   by $\Phi_{t+1/2}$.

\STATE \textbf{Elitist Selection:} if $f(\Phi_{t+1/2}) >f(\Phi_t)$, then let $\Phi_{t+1} \leftarrow \Phi_{t+1/2}$; otherwise $\Phi_{t+1} \leftarrow\Phi_{t}$.
\end{algorithmic}
\end{algorithm}
 
Th sequence $\{\Phi_t; t=0,1, \cdots \}$ is a Markov chain and its transition probabilities are given as follows:  for any $x,y \in \{0, 1, \cdots, 100 \}$ which are not   in the optimal solution set, 
\begin{align*}
P(x,y) =\left\{
\begin{array}{llll}
0.01, &\mbox{if } y=x-1   \mbox{ and } f(y)>f(x);\\
0.01, &\mbox{if } y=x+1   \mbox{ and } f(y)>f(x);\\
p, &\mbox{if } y=x;\\
0, &\mbox{otherwise}.
\end{array}
\right.
\end{align*}
where
$p=1-P(x,x-1)-P(x,x+1)$.

The second   algorithm  adopts random walk with non-elitist selection, denoted by RSH-II (see Algorithm \ref{alg3}). Different from RSI-I, RSH-II allows a worse child to be accepted.  

\begin{algorithm}[ht] 
\caption{RSH-II} 
\label{alg3}
\begin{algorithmic} 
\STATE   \textbf{Random Walk}: provided that $\Phi_t=x$, then  it walks  to $x-1$ (if $x-1$ is not less than $0$) with  probability $0.01$, or walks   to $x+1$ (if $x+1$ is not more than $100$)  with  probability $0.01$. Denote the new position   by $\Phi_{t+1/2}$.

\STATE \textbf{Non-elitist Selection:} if $f(\Phi_{t+1/2}) >f(\Phi_t)$, then definitely let $\Phi_{t+1} \leftarrow \Phi_{t+1/2}$; otherwise let $\Phi_{t+1} \leftarrow\Phi_{t+1/2}$ with  probability $0.5$.
\end{algorithmic}
\end{algorithm}

Th sequence $\{\Phi_t; t=0,1, \cdots \}$ is a Markov chain and its transition probabilities are given as follows: for any $x,y \in \{0, 1, \cdots, 100 \}$ which are not   in the optimal solution set, 
\begin{align*}
P(x,y) =\left\{
\begin{array}{llll}
0.01, &\mbox{if } y=x-1   \mbox{ and } f(y)>f(x);\\
0.005, &\mbox{if } y=x-1   \mbox{ and } f(y) \le f(x);\\
0.01, &\mbox{if } y=x+1   \mbox{ and } f(y)>f(x);\\
0.005, &\mbox{if } y=x+1   \mbox{ and } f(y) \le f(x);\\
p, &\mbox{if } y=x;\\
0, &\mbox{otherwise};
\end{array}
\right.
\end{align*}
where
$p=1-P(x,x-1)-P(x,x+1)$.

\section{Convergence in Distribution}
\label{secConvergence}

\subsection{Theory Study: Convergence Condition} 
In this subsection we define the convergence  in distribution  of randomised search heuristics and  establish the sufficient and necessary conditions of convergence in distribution.

Let $S_{\opt}$ denote the set of all populations which includes at least one  optimal solution, and $S_{\non}$ the set of all populations which doesn't include any optimal solution.
\begin{definition}
A  randomised search heuristic  is called \emph{convergence in distribution}  if   starting from any non-optimal population, the probability of $\Phi_t$ in the  optimal set  goes towards  $1$ as $t$ to the infinitely large. That is 
$$
\lim_{t \to +\infty} P(\Phi_t \in S_{\opt}) = 1.
$$
\end{definition}

We consider the probability  distribution of $\Phi_t$  in the non-optimal set. Let $q_t(X)$ denote the probability of $\Phi_t$ at a non-optimal state $X$, 
$$
q_t(X): =P(\Phi_t=X).
$$
Let  $(X_1, X_2, \cdots )$ represent all populations in the non-optimal set. Then
the vector\footnote{Notation  $\mathbf{v}$ represents a column vector and $\mathbf{v}^T$ the row column with the transpose operation.} $\mathbf{q}_t$  
$$\mathbf{q}_t:= (q_t(X_1), q_t(X_2), \cdots   )^T
$$ 
  denote   the probability distribution of  $\Phi_t$ over all non-optimal populations.

 Notice that   the  vector 1-norm\footnote{Given an $n \times n$   matrix $\mathbf{A}=[a_{i,j}]$, its $1$-norm is    $\parallel \mathbf{A} \parallel_1 = \max_{1 \le j \le n}  \sum _{i=1}^n \mid a_{ij} \mid$. Given a vector $\mathbf{v}$, its $1$-norm is
$\sum_{i=1}^n \mid v_i\mid $.} $\parallel \mathbf{q}_t  \parallel_1$  equals to  the population $\Phi_t$ in the non-optimal set, i.e.,   
$$ \parallel \mathbf{q}_t  \parallel_1 = P(  \Phi_{t} \in S_{\non}). 
$$
Thus the convergence of a randomised search heuristic  is rewritten in the norm form:
$$
\lim_{t \to +\infty}  \parallel \mathbf{q}_t  \parallel_1 = 0.
$$ 

The 1-norm  $\parallel \mathbf{q}_t  \parallel_1$   plays the role of the distance   between the population $\Phi_t$ and the optimal set.   
  This can be viewed as a special case of the analysis in \cite{he1999convergence}, where  a general vector norm has been used as the distance.

In the following we draw the convergence condition of randomised search heuristics based on the absorbing Markov chain theory. 

Let $\mathbf{P}$ denote the transition matrix  of the Markov chain associated with a randomised search heuristic, whose  entries are $P(X,Y),$ where $ X, Y \in S$.
Since a state in the optimal set  is always absorbing, so the  transition matrix $\mathbf{P} $  can be written in the   canonical form below,
\begin{equation}
\label{equCanonicalForm} 
\mathbf{P} = 
\begin{pmatrix}
  \mathbf{I} & \mathbf{O}\\
 \mathbf{R} & \mathbf{Q}\\
 \end{pmatrix},
\end{equation}
where $\mathbf{I}$ is a unit matrix and $\mathbf{O}$ a zero matrix.  $ \mathbf{Q}$ is a matrix to denote    probability transitions among non-optimal  populations.
  $\mathbf{R}$ is a matrix to represent  probability transitions from non-optimal populations to the optimal set.

According to    absorbing Markov chain theory~\cite[Chapter 11]{grinstead1997introduction},  
\begin{align*}
P(\Phi_{t+1}=Y) =\sum_{X \in S_{\non}} P(\Phi_t=X) P(X,Y),
\end{align*}
then  the  iteration 
$\Phi_t \to \Phi_{t+1}$  is represented in an  equivalent   matrix iteration,
\begin{equation}
\label{equMatrixIteration}
\mathbf{q}^T_{t+1}=   \mathbf{q}^T_t \mathbf{Q} =\mathbf{q}^T_0 \mathbf{Q}^t  .  
\end{equation}

From the sufficient and necessary condition of convergence of iterative methods (\cite[Theorem 1.10]{varga2009matrix}), it is straightforward to obtain the  convergence condition for  randomised search heuristics:
\begin{lemma}\label{lemConvergence} 
A randomised search heuristic  is convergent 
 if and only if the spectral radius\footnote{The spectral radius of a  square matrix  $\mathbf{A}$, denoted by $\rho(\mathbf{A})$, is  the supremum among the absolute values of all eigenvalues of $\mathbf{A}.$ }  $\rho(\mathbf{Q})<1.$
\end{lemma}

The above sufficient and necessary condition is less useful in practice since it is too difficult to calculate the spectral radius   of the  transition matrix. Therefore we turn to find an equivalent condition which is much  easier to verify.  The following lemma gives such a condition.

\begin{lemma}
\label{lemConvergence2}
A  randomised search heuristic is convergent 
 if and only if   starting from any non-optimal population, it is possible to visit the optimal set after finite iterations. That is, there exists an integer  $k>0$ and for any non-optimal state $X$ and $t \ge 0$,   
\begin{align}
\label{equConvergence}
   P(  \Phi_{t+k} \in S_{\opt} \mid \Phi_t =X)
>0.
\end{align}
\end{lemma}

\begin{proof}
Since the chain $\{ \Phi_t; t =0,1, \cdots\}$ is homogeneous, thus it is enough to prove  the case of $t=0$.

(i) The proof that the condition is sufficient.

Assume  (\ref{equConvergence}) holds. From 
$$
 P(  \Phi_{k} \in S_{\non} \mid \Phi_0 =X) +  P(  \Phi_{k} \in S_{\opt} \mid \Phi_0 =X) =1, 
$$
then it gives
$
 \parallel \mathbf{q}_k  \parallel_1 
 <  1
$ for any $ \parallel \mathbf{q}_0  \parallel_1 $ such that  $\parallel \mathbf{q}_0  \parallel_1 
 = 1$.

Then from the matrix iteration
$
 \mathbf{q}_k    =(\mathbf{Q}^T)^k \mathbf{q}_0,  
$
and the matrix 1-norm definition
\begin{align*}
\parallel (\mathbf{Q}^T)^k \parallel_1 
 =\max_{\parallel \mathbf{q}_0  \parallel_1 
 = 1} \parallel (\mathbf{Q}^T)^k \mathbf{q}_0 \parallel_1,
\end{align*} 
it follows
\begin{align*}
 \parallel (\mathbf{Q}^T)^k \parallel_1 
 <  1.
\end{align*}

Since the spectral radius of a matrix is never bigger than its consistent norm  \citep[Example 7.1.4]{meyer2000matrix}, the above inequality yields
\[
 \rho\left( (\mathbf{Q}^T)^k\right) \le \parallel  (\mathbf{Q}^T)^k \parallel_1 < 1,
\]
so that
$
 \rho(\mathbf{Q}) =\rho(\mathbf{Q}^T) <  1.
$
According to Lemma~\ref{lemConvergence}, this implies  that the algorithm is convergent.

(ii) The proof that the condition is necessary.

Suppose the algorithm is convergent, then according to Lemma \ref{lemConvergence}, $\rho(\mathbf{Q})=\rho(\mathbf{Q}^T)<1$. From Gelfand's spectral radius   formula\footnote{Gelfand's spectral radius   formula says that for any induced matrix norm: $\lim_{t \to +\infty }\parallel \mathbf{A}^k \parallel^{1/k}=\rho(\mathbf{A})$   \citep[Example 7.10.1]{meyer2000matrix}.}, there exists an integer $k>0$ such that 
\begin{equation}
\label{equMaximumNorm1}
\parallel (\mathbf{Q}^T)^k \parallel^{1/k}_1 <1.
\end{equation}
and then $\parallel (\mathbf{Q}^T)^k \parallel_1 <1.$

From  $
 \mathbf{q}_k    =(\mathbf{Q}^T)^k \mathbf{q}_0,  
$
it gives
$
 \parallel \mathbf{q}_k  \parallel_1 
 <  1.
$

From 
$$
 P(  \Phi_{k} \in S_{\non} \mid \Phi_0 =X) +  P(  \Phi_{k} \in S_{\opt} \mid \Phi_0 =X) =1, 
$$
and  
$$ \parallel \mathbf{q}_k  \parallel_1 = P(  \Phi_{k} \in S_{\non} \mid \Phi_0 =X), 
$$
It follows then that
$$
   P(  \Phi_k \in S_{\opt} \mid \Phi_0 =X)
>0,$$ 
and this proves (\ref{equConvergence}).
\end{proof}

From the above lemma, we can easily draw the following sufficient and necessary condition, which is much easier to verify in practice.
 \begin{theorem}
\label{theConvergence}
A  randomised search heuristic is convergent 
 if and only if   starting from any non-optimal state, it is possible to reach  a better state after finite iterations. That is, there exists an integer  $k>0$ and for any non-optimal state $X$ and   $t \ge 0$,   
\begin{align}
   P( f_{t+k} >f_t \mid \Phi_t =X)>0.
\end{align}
\end{theorem} 

\subsection{Case Study}
In this subsection, we show how the convergence conditions are applied to determine the convergence of a randomised search heuristic.

\begin{example}
Consider RSH-I and RSH-II for solving the following maximisation problem, 
\begin{align}
& \max  \, f(x), &x \in \{ 0,1, \cdots, 100 \}.
\end{align} 
\end{example}

RSH-I does not converge if there exists a state $x$ such that $f(x)>f(x-1)$ and $f(x)>f(x+1)$. This means $f(x)$ is a multi-modal function. At the local optimum $x$, the algorithm cannot make any move.

RSH-II always converges. It is easy to verify that from any state $X$, the algorithm can make an improvement at most 100 iterations with a positive probability, thus according to  Theorem~\ref{theConvergence}, the algorithm is not convergent.
 
\begin{example}
Consider RSH-I for solving the following maximisation problem  
\begin{align}
& \max  \, x^2, &x \in \{ 0,1, \cdots, 100 \}.
\end{align} 
It is a unimodal function with the optimum at $100$.
\end{example}

The corresponding transition matrix   $\mathbf{Q}$ is
\begin{equation}
\mathbf{Q} =  \begin{pmatrix} 
 0.99 & 0 & \cdots & 0 & 0\\
  0.01 & 0.99& \cdots & 0 & 0\\
  \vdots & \vdots & \vdots& \vdots & \vdots  \\
  0 & 0& \cdots & 0.99 & 0\\
  0 & 0& \cdots & 0.01 & 0.99\\
 \end{pmatrix}.
\end{equation}
 
We can prove the convergence of RSH-I using Lemma~\ref{lemConvergence}. Since $\rho(\mathbf{Q})=0.99<1$, so RSH-I for maximising $f_1(x)$ is convergent. 

We also can prove   the convergence of RSH-I using Theorem~\ref{theConvergence} without calculating the spectral radius. It is easy to see that the probability of obtaining a better child is   $0.01$, thus according to  Theorem~\ref{theConvergence}, the algorithm is convergent.

Comparing with the two approaches,  we see that Theorem~\ref{theConvergence} provides a faster approach to determine the convergence.

\subsection{Computational Study}
In this subsection, we propose a computational approach to illustrating  the convergence in distribution of randomised search heuristics. 

In the computation study, we run an algorithm for $k$ times. Let $n (\Phi_t \in S_{\opt})$ denote the number of  $\Phi_t$  (where $t=0,1, \cdots $) appearing in the optimal set for these $k$ runs.  

According to the law of large numbers, the probability $P(\Phi_t \in S_{\opt})$  will be approximated by the relative frequency as follows:
$$
\frac{n (\Phi_t \in S_{\opt})}{k}, \mbox{ when } k \to +\infty.
$$

The above frequency is used  as  the probability $P(\Phi_t \in S_{\opt})$ in the the computational study.

\begin{example}
Consider RSH-II for solving the following problem   
\begin{align}
&\max \, (x-49)^2, \quad x \in \{ 0,1, \cdots, 100 \}.
\end{align}
It is a two-modal function, with one local optimum at $0$ and one global optimum at  $100$. 
\end{example}

We apply   RSH-I to the  problem. 
Run the algorithm for 100,000 times. The initial population is $\Phi_0=20$.  

  Figure~\ref{f3-convergence2} shows the probability of $\Phi_t$ in the optimal set is 0 when $t\le 100,000$. In other words, no convergence happens yet in $100,000$ iterations.  
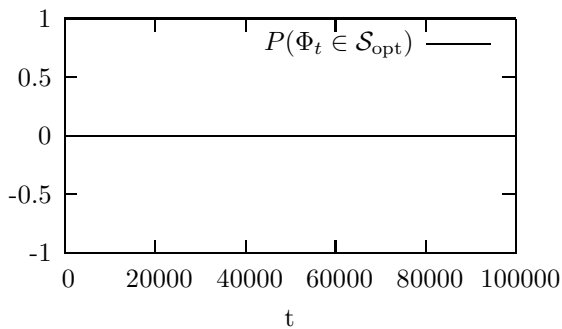
\begin{figure}[ht]
\begin{center}
\setlength{\unitlength}{0.240900pt}
\ifx\plotpoint\undefined\newsavebox{\plotpoint}\fi
\begin{picture}(900,540)(0,0)
\sbox{\plotpoint}{\rule[-0.200pt]{0.400pt}{0.400pt}}%
\put(130.0,131.0){\rule[-0.200pt]{4.818pt}{0.400pt}}
\put(110,131){\makebox(0,0)[r]{-1}}
\put(819.0,131.0){\rule[-0.200pt]{4.818pt}{0.400pt}}
\put(130.0,223.0){\rule[-0.200pt]{4.818pt}{0.400pt}}
\put(110,223){\makebox(0,0)[r]{-0.5}}
\put(819.0,223.0){\rule[-0.200pt]{4.818pt}{0.400pt}}
\put(130.0,315.0){\rule[-0.200pt]{4.818pt}{0.400pt}}
\put(110,315){\makebox(0,0)[r]{ 0}}
\put(819.0,315.0){\rule[-0.200pt]{4.818pt}{0.400pt}}
\put(130.0,407.0){\rule[-0.200pt]{4.818pt}{0.400pt}}
\put(110,407){\makebox(0,0)[r]{ 0.5}}
\put(819.0,407.0){\rule[-0.200pt]{4.818pt}{0.400pt}}
\put(130.0,499.0){\rule[-0.200pt]{4.818pt}{0.400pt}}
\put(110,499){\makebox(0,0)[r]{ 1}}
\put(819.0,499.0){\rule[-0.200pt]{4.818pt}{0.400pt}}
\put(130.0,131.0){\rule[-0.200pt]{0.400pt}{4.818pt}}
\put(130,90){\makebox(0,0){ 0}}
\put(130.0,479.0){\rule[-0.200pt]{0.400pt}{4.818pt}}
\put(272.0,131.0){\rule[-0.200pt]{0.400pt}{4.818pt}}
\put(272,90){\makebox(0,0){ 20000}}
\put(272.0,479.0){\rule[-0.200pt]{0.400pt}{4.818pt}}
\put(414.0,131.0){\rule[-0.200pt]{0.400pt}{4.818pt}}
\put(414,90){\makebox(0,0){ 40000}}
\put(414.0,479.0){\rule[-0.200pt]{0.400pt}{4.818pt}}
\put(555.0,131.0){\rule[-0.200pt]{0.400pt}{4.818pt}}
\put(555,90){\makebox(0,0){ 60000}}
\put(555.0,479.0){\rule[-0.200pt]{0.400pt}{4.818pt}}
\put(697.0,131.0){\rule[-0.200pt]{0.400pt}{4.818pt}}
\put(697,90){\makebox(0,0){ 80000}}
\put(697.0,479.0){\rule[-0.200pt]{0.400pt}{4.818pt}}
\put(839.0,131.0){\rule[-0.200pt]{0.400pt}{4.818pt}}
\put(839,90){\makebox(0,0){ 100000}}
\put(839.0,479.0){\rule[-0.200pt]{0.400pt}{4.818pt}}
\put(130.0,131.0){\rule[-0.200pt]{0.400pt}{88.651pt}}
\put(130.0,131.0){\rule[-0.200pt]{170.798pt}{0.400pt}}
\put(839.0,131.0){\rule[-0.200pt]{0.400pt}{88.651pt}}
\put(130.0,499.0){\rule[-0.200pt]{170.798pt}{0.400pt}}
\put(484,29){\makebox(0,0){t}}
\put(679,459){\makebox(0,0)[r]{$P(\Phi_t \in  \mathcal{S}_{\opt})$}}
\put(699.0,459.0){\rule[-0.200pt]{24.090pt}{0.400pt}}
\put(130,315){\usebox{\plotpoint}}
\put(130,315){\usebox{\plotpoint}}
\put(130,315){\usebox{\plotpoint}}
\put(130,315){\usebox{\plotpoint}}
\put(130,315){\usebox{\plotpoint}}
\put(130,315){\usebox{\plotpoint}}
\put(130,315){\usebox{\plotpoint}}
\put(130,315){\usebox{\plotpoint}}
\put(130,315){\usebox{\plotpoint}}
\put(130,315){\usebox{\plotpoint}}
\put(130,315){\usebox{\plotpoint}}
\put(130,315){\usebox{\plotpoint}}
\put(130,315){\usebox{\plotpoint}}
\put(130,315){\usebox{\plotpoint}}
\put(130,315){\usebox{\plotpoint}}
\put(130,315){\usebox{\plotpoint}}
\put(130,315){\usebox{\plotpoint}}
\put(130,315){\usebox{\plotpoint}}
\put(130,315){\usebox{\plotpoint}}
\put(130,315){\usebox{\plotpoint}}
\put(130,315){\usebox{\plotpoint}}
\put(130,315){\usebox{\plotpoint}}
\put(130,315){\usebox{\plotpoint}}
\put(130,315){\usebox{\plotpoint}}
\put(130,315){\usebox{\plotpoint}}
\put(130,315){\usebox{\plotpoint}}
\put(130,315){\usebox{\plotpoint}}
\put(130,315){\usebox{\plotpoint}}
\put(130,315){\usebox{\plotpoint}}
\put(130,315){\usebox{\plotpoint}}
\put(130,315){\usebox{\plotpoint}}
\put(130,315){\usebox{\plotpoint}}
\put(130,315){\usebox{\plotpoint}}
\put(130,315){\usebox{\plotpoint}}
\put(130,315){\usebox{\plotpoint}}
\put(130,315){\usebox{\plotpoint}}
\put(130,315){\usebox{\plotpoint}}
\put(130,315){\usebox{\plotpoint}}
\put(130,315){\usebox{\plotpoint}}
\put(130,315){\usebox{\plotpoint}}
\put(130,315){\usebox{\plotpoint}}
\put(130,315){\usebox{\plotpoint}}
\put(130,315){\usebox{\plotpoint}}
\put(130,315){\usebox{\plotpoint}}
\put(130,315){\usebox{\plotpoint}}
\put(130,315){\usebox{\plotpoint}}
\put(130,315){\usebox{\plotpoint}}
\put(130,315){\usebox{\plotpoint}}
\put(130,315){\usebox{\plotpoint}}
\put(130,315){\usebox{\plotpoint}}
\put(130,315){\usebox{\plotpoint}}
\put(130,315){\usebox{\plotpoint}}
\put(130,315){\usebox{\plotpoint}}
\put(130,315){\usebox{\plotpoint}}
\put(130,315){\usebox{\plotpoint}}
\put(130,315){\usebox{\plotpoint}}
\put(130,315){\usebox{\plotpoint}}
\put(130,315){\usebox{\plotpoint}}
\put(130,315){\usebox{\plotpoint}}
\put(130,315){\usebox{\plotpoint}}
\put(130,315){\usebox{\plotpoint}}
\put(130,315){\usebox{\plotpoint}}
\put(130,315){\usebox{\plotpoint}}
\put(130,315){\usebox{\plotpoint}}
\put(130,315){\usebox{\plotpoint}}
\put(130,315){\usebox{\plotpoint}}
\put(130,315){\usebox{\plotpoint}}
\put(130,315){\usebox{\plotpoint}}
\put(130,315){\usebox{\plotpoint}}
\put(130,315){\usebox{\plotpoint}}
\put(130,315){\usebox{\plotpoint}}
\put(130.0,315.0){\rule[-0.200pt]{170.798pt}{0.400pt}}
\put(130.0,131.0){\rule[-0.200pt]{0.400pt}{88.651pt}}
\put(130.0,131.0){\rule[-0.200pt]{170.798pt}{0.400pt}}
\put(839.0,131.0){\rule[-0.200pt]{0.400pt}{88.651pt}}
\put(130.0,499.0){\rule[-0.200pt]{170.798pt}{0.400pt}}
\end{picture}
\end{center}
\caption{The probability   $P(\Phi_t \in S_{\opt})$ when RSH-I  maximises $(x-49)^2$.}
\label{f3-convergence2} 
\end{figure}

Our approach  is different from  that of visualising the fitness value over $t$. The latter approach  may be the most popular used  for illustrating the convergence of randomised search heuristics (for example, see Figures 1 to 3 in \cite{back1993overview}).   
Figure~\ref{f3-convergence1} shows   the fitness value $f_t$ `converges' after about 5,000 iterations  and  thereafter no change.  
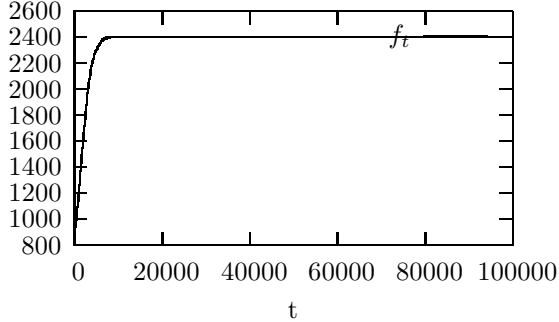
\begin{figure}[ht]
\begin{center}
\setlength{\unitlength}{0.240900pt}
\ifx\plotpoint\undefined\newsavebox{\plotpoint}\fi
\begin{picture}(900,540)(0,0)
\sbox{\plotpoint}{\rule[-0.200pt]{0.400pt}{0.400pt}}%
\put(150.0,131.0){\rule[-0.200pt]{4.818pt}{0.400pt}}
\put(130,131){\makebox(0,0)[r]{ 800}}
\put(819.0,131.0){\rule[-0.200pt]{4.818pt}{0.400pt}}
\put(150.0,172.0){\rule[-0.200pt]{4.818pt}{0.400pt}}
\put(130,172){\makebox(0,0)[r]{ 1000}}
\put(819.0,172.0){\rule[-0.200pt]{4.818pt}{0.400pt}}
\put(150.0,213.0){\rule[-0.200pt]{4.818pt}{0.400pt}}
\put(130,213){\makebox(0,0)[r]{ 1200}}
\put(819.0,213.0){\rule[-0.200pt]{4.818pt}{0.400pt}}
\put(150.0,254.0){\rule[-0.200pt]{4.818pt}{0.400pt}}
\put(130,254){\makebox(0,0)[r]{ 1400}}
\put(819.0,254.0){\rule[-0.200pt]{4.818pt}{0.400pt}}
\put(150.0,295.0){\rule[-0.200pt]{4.818pt}{0.400pt}}
\put(130,295){\makebox(0,0)[r]{ 1600}}
\put(819.0,295.0){\rule[-0.200pt]{4.818pt}{0.400pt}}
\put(150.0,335.0){\rule[-0.200pt]{4.818pt}{0.400pt}}
\put(130,335){\makebox(0,0)[r]{ 1800}}
\put(819.0,335.0){\rule[-0.200pt]{4.818pt}{0.400pt}}
\put(150.0,376.0){\rule[-0.200pt]{4.818pt}{0.400pt}}
\put(130,376){\makebox(0,0)[r]{ 2000}}
\put(819.0,376.0){\rule[-0.200pt]{4.818pt}{0.400pt}}
\put(150.0,417.0){\rule[-0.200pt]{4.818pt}{0.400pt}}
\put(130,417){\makebox(0,0)[r]{ 2200}}
\put(819.0,417.0){\rule[-0.200pt]{4.818pt}{0.400pt}}
\put(150.0,458.0){\rule[-0.200pt]{4.818pt}{0.400pt}}
\put(130,458){\makebox(0,0)[r]{ 2400}}
\put(819.0,458.0){\rule[-0.200pt]{4.818pt}{0.400pt}}
\put(150.0,499.0){\rule[-0.200pt]{4.818pt}{0.400pt}}
\put(130,499){\makebox(0,0)[r]{ 2600}}
\put(819.0,499.0){\rule[-0.200pt]{4.818pt}{0.400pt}}
\put(150.0,131.0){\rule[-0.200pt]{0.400pt}{4.818pt}}
\put(150,90){\makebox(0,0){ 0}}
\put(150.0,479.0){\rule[-0.200pt]{0.400pt}{4.818pt}}
\put(288.0,131.0){\rule[-0.200pt]{0.400pt}{4.818pt}}
\put(288,90){\makebox(0,0){ 20000}}
\put(288.0,479.0){\rule[-0.200pt]{0.400pt}{4.818pt}}
\put(426.0,131.0){\rule[-0.200pt]{0.400pt}{4.818pt}}
\put(426,90){\makebox(0,0){ 40000}}
\put(426.0,479.0){\rule[-0.200pt]{0.400pt}{4.818pt}}
\put(563.0,131.0){\rule[-0.200pt]{0.400pt}{4.818pt}}
\put(563,90){\makebox(0,0){ 60000}}
\put(563.0,479.0){\rule[-0.200pt]{0.400pt}{4.818pt}}
\put(701.0,131.0){\rule[-0.200pt]{0.400pt}{4.818pt}}
\put(701,90){\makebox(0,0){ 80000}}
\put(701.0,479.0){\rule[-0.200pt]{0.400pt}{4.818pt}}
\put(839.0,131.0){\rule[-0.200pt]{0.400pt}{4.818pt}}
\put(839,90){\makebox(0,0){ 100000}}
\put(839.0,479.0){\rule[-0.200pt]{0.400pt}{4.818pt}}
\put(150.0,131.0){\rule[-0.200pt]{0.400pt}{88.651pt}}
\put(150.0,131.0){\rule[-0.200pt]{165.980pt}{0.400pt}}
\put(839.0,131.0){\rule[-0.200pt]{0.400pt}{88.651pt}}
\put(150.0,499.0){\rule[-0.200pt]{165.980pt}{0.400pt}}
\put(494,29){\makebox(0,0){t}}
\put(679,459){\makebox(0,0)[r]{$f_t$}}
\put(699.0,459.0){\rule[-0.200pt]{24.090pt}{0.400pt}}
\put(150,139){\usebox{\plotpoint}}
\put(150,139){\usebox{\plotpoint}}
\put(150.0,139.0){\rule[-0.200pt]{0.400pt}{1.927pt}}
\put(150.0,147.0){\usebox{\plotpoint}}
\put(151.0,147.0){\rule[-0.200pt]{0.400pt}{3.132pt}}
\put(151.0,160.0){\usebox{\plotpoint}}
\put(152,169.67){\rule{0.241pt}{0.400pt}}
\multiput(152.00,169.17)(0.500,1.000){2}{\rule{0.120pt}{0.400pt}}
\put(152.0,160.0){\rule[-0.200pt]{0.400pt}{2.409pt}}
\put(153,171){\usebox{\plotpoint}}
\put(153,171){\usebox{\plotpoint}}
\put(153,171){\usebox{\plotpoint}}
\put(153,171){\usebox{\plotpoint}}
\put(153,171){\usebox{\plotpoint}}
\put(153,171){\usebox{\plotpoint}}
\put(153,171){\usebox{\plotpoint}}
\put(153,171){\usebox{\plotpoint}}
\put(153,171){\usebox{\plotpoint}}
\put(153,171){\usebox{\plotpoint}}
\put(153,171){\usebox{\plotpoint}}
\put(153,171){\usebox{\plotpoint}}
\put(153,171){\usebox{\plotpoint}}
\put(153.0,171.0){\rule[-0.200pt]{0.400pt}{2.409pt}}
\put(153.0,181.0){\usebox{\plotpoint}}
\put(154.0,181.0){\rule[-0.200pt]{0.400pt}{2.409pt}}
\put(154.0,191.0){\usebox{\plotpoint}}
\put(155,199.67){\rule{0.241pt}{0.400pt}}
\multiput(155.00,199.17)(0.500,1.000){2}{\rule{0.120pt}{0.400pt}}
\put(155.0,191.0){\rule[-0.200pt]{0.400pt}{2.168pt}}
\put(156,201){\usebox{\plotpoint}}
\put(156,201){\usebox{\plotpoint}}
\put(156,201){\usebox{\plotpoint}}
\put(156,201){\usebox{\plotpoint}}
\put(156,201){\usebox{\plotpoint}}
\put(156,201){\usebox{\plotpoint}}
\put(156,201){\usebox{\plotpoint}}
\put(156,201){\usebox{\plotpoint}}
\put(156,201){\usebox{\plotpoint}}
\put(156,201){\usebox{\plotpoint}}
\put(156,201){\usebox{\plotpoint}}
\put(156.0,201.0){\rule[-0.200pt]{0.400pt}{2.891pt}}
\put(156.0,213.0){\usebox{\plotpoint}}
\put(157.0,213.0){\rule[-0.200pt]{0.400pt}{2.891pt}}
\put(157.0,225.0){\usebox{\plotpoint}}
\put(158,235.67){\rule{0.241pt}{0.400pt}}
\multiput(158.00,235.17)(0.500,1.000){2}{\rule{0.120pt}{0.400pt}}
\put(158.0,225.0){\rule[-0.200pt]{0.400pt}{2.650pt}}
\put(159,237){\usebox{\plotpoint}}
\put(159,237){\usebox{\plotpoint}}
\put(159,237){\usebox{\plotpoint}}
\put(159,237){\usebox{\plotpoint}}
\put(159,237){\usebox{\plotpoint}}
\put(159,237){\usebox{\plotpoint}}
\put(159,237){\usebox{\plotpoint}}
\put(159,237){\usebox{\plotpoint}}
\put(159,237){\usebox{\plotpoint}}
\put(159,237){\usebox{\plotpoint}}
\put(159,237){\usebox{\plotpoint}}
\put(159.0,237.0){\rule[-0.200pt]{0.400pt}{2.650pt}}
\put(159.0,248.0){\usebox{\plotpoint}}
\put(160.0,248.0){\rule[-0.200pt]{0.400pt}{2.891pt}}
\put(160.0,260.0){\usebox{\plotpoint}}
\put(161.0,260.0){\rule[-0.200pt]{0.400pt}{2.891pt}}
\put(161.0,272.0){\usebox{\plotpoint}}
\put(162.0,272.0){\rule[-0.200pt]{0.400pt}{2.891pt}}
\put(162.0,284.0){\usebox{\plotpoint}}
\put(163.0,284.0){\rule[-0.200pt]{0.400pt}{2.891pt}}
\put(163.0,296.0){\usebox{\plotpoint}}
\put(164,306.67){\rule{0.241pt}{0.400pt}}
\multiput(164.00,306.17)(0.500,1.000){2}{\rule{0.120pt}{0.400pt}}
\put(164.0,296.0){\rule[-0.200pt]{0.400pt}{2.650pt}}
\put(165,308){\usebox{\plotpoint}}
\put(165,308){\usebox{\plotpoint}}
\put(165,308){\usebox{\plotpoint}}
\put(165,308){\usebox{\plotpoint}}
\put(165,308){\usebox{\plotpoint}}
\put(165,308){\usebox{\plotpoint}}
\put(165,308){\usebox{\plotpoint}}
\put(165,308){\usebox{\plotpoint}}
\put(165,308){\usebox{\plotpoint}}
\put(165,308){\usebox{\plotpoint}}
\put(165,308){\usebox{\plotpoint}}
\put(165,308){\usebox{\plotpoint}}
\put(165.0,308.0){\rule[-0.200pt]{0.400pt}{2.650pt}}
\put(165.0,319.0){\usebox{\plotpoint}}
\put(166.0,319.0){\rule[-0.200pt]{0.400pt}{2.650pt}}
\put(166.0,330.0){\usebox{\plotpoint}}
\put(167.0,330.0){\rule[-0.200pt]{0.400pt}{2.409pt}}
\put(167.0,340.0){\usebox{\plotpoint}}
\put(168.0,340.0){\rule[-0.200pt]{0.400pt}{2.650pt}}
\put(168.0,351.0){\usebox{\plotpoint}}
\put(169.0,351.0){\rule[-0.200pt]{0.400pt}{2.168pt}}
\put(169.0,360.0){\usebox{\plotpoint}}
\put(170,368.67){\rule{0.241pt}{0.400pt}}
\multiput(170.00,368.17)(0.500,1.000){2}{\rule{0.120pt}{0.400pt}}
\put(170.0,360.0){\rule[-0.200pt]{0.400pt}{2.168pt}}
\put(171,370){\usebox{\plotpoint}}
\put(171,370){\usebox{\plotpoint}}
\put(171,370){\usebox{\plotpoint}}
\put(171,370){\usebox{\plotpoint}}
\put(171,370){\usebox{\plotpoint}}
\put(171,370){\usebox{\plotpoint}}
\put(171,370){\usebox{\plotpoint}}
\put(171,370){\usebox{\plotpoint}}
\put(171,370){\usebox{\plotpoint}}
\put(171,370){\usebox{\plotpoint}}
\put(171,370){\usebox{\plotpoint}}
\put(171,370){\usebox{\plotpoint}}
\put(171.0,370.0){\rule[-0.200pt]{0.400pt}{1.927pt}}
\put(171.0,378.0){\usebox{\plotpoint}}
\put(172.0,378.0){\rule[-0.200pt]{0.400pt}{1.927pt}}
\put(172.0,386.0){\usebox{\plotpoint}}
\put(173.0,386.0){\rule[-0.200pt]{0.400pt}{1.445pt}}
\put(173.0,392.0){\usebox{\plotpoint}}
\put(174.0,392.0){\rule[-0.200pt]{0.400pt}{1.686pt}}
\put(174.0,399.0){\usebox{\plotpoint}}
\put(175.0,399.0){\rule[-0.200pt]{0.400pt}{1.686pt}}
\put(175.0,406.0){\usebox{\plotpoint}}
\put(176.0,406.0){\rule[-0.200pt]{0.400pt}{1.445pt}}
\put(176.0,412.0){\usebox{\plotpoint}}
\put(177.0,412.0){\rule[-0.200pt]{0.400pt}{1.204pt}}
\put(177.0,417.0){\usebox{\plotpoint}}
\put(178.0,417.0){\rule[-0.200pt]{0.400pt}{0.723pt}}
\put(178.0,420.0){\usebox{\plotpoint}}
\put(179.0,420.0){\rule[-0.200pt]{0.400pt}{0.964pt}}
\put(179.0,424.0){\usebox{\plotpoint}}
\put(180.0,424.0){\rule[-0.200pt]{0.400pt}{0.964pt}}
\put(180.0,428.0){\usebox{\plotpoint}}
\put(181.0,428.0){\rule[-0.200pt]{0.400pt}{0.723pt}}
\put(181.0,431.0){\usebox{\plotpoint}}
\put(182.0,431.0){\rule[-0.200pt]{0.400pt}{0.723pt}}
\put(182.0,434.0){\usebox{\plotpoint}}
\put(183.0,434.0){\rule[-0.200pt]{0.400pt}{0.482pt}}
\put(183.0,436.0){\usebox{\plotpoint}}
\put(184.0,436.0){\rule[-0.200pt]{0.400pt}{0.482pt}}
\put(184.0,438.0){\usebox{\plotpoint}}
\put(185.0,438.0){\rule[-0.200pt]{0.400pt}{0.482pt}}
\put(185.0,440.0){\usebox{\plotpoint}}
\put(186.0,440.0){\rule[-0.200pt]{0.400pt}{0.723pt}}
\put(186.0,443.0){\usebox{\plotpoint}}
\put(187.0,443.0){\usebox{\plotpoint}}
\put(187.0,444.0){\usebox{\plotpoint}}
\put(188.0,444.0){\rule[-0.200pt]{0.400pt}{0.482pt}}
\put(188.0,446.0){\usebox{\plotpoint}}
\put(189.0,446.0){\usebox{\plotpoint}}
\put(189.0,447.0){\usebox{\plotpoint}}
\put(190.0,447.0){\rule[-0.200pt]{0.400pt}{0.482pt}}
\put(190.0,449.0){\usebox{\plotpoint}}
\put(191.0,449.0){\rule[-0.200pt]{0.400pt}{0.482pt}}
\put(191.0,451.0){\usebox{\plotpoint}}
\put(192.0,451.0){\usebox{\plotpoint}}
\put(192.0,452.0){\usebox{\plotpoint}}
\put(193.0,452.0){\usebox{\plotpoint}}
\put(193.0,453.0){\usebox{\plotpoint}}
\put(194.0,453.0){\usebox{\plotpoint}}
\put(194.0,454.0){\rule[-0.200pt]{0.482pt}{0.400pt}}
\put(196.0,454.0){\usebox{\plotpoint}}
\put(196.0,455.0){\usebox{\plotpoint}}
\put(197.0,455.0){\usebox{\plotpoint}}
\put(197.0,456.0){\rule[-0.200pt]{0.723pt}{0.400pt}}
\put(200.0,456.0){\usebox{\plotpoint}}
\put(200.0,457.0){\rule[-0.200pt]{1.445pt}{0.400pt}}
\put(206.0,457.0){\usebox{\plotpoint}}
\put(206.0,458.0){\rule[-0.200pt]{152.490pt}{0.400pt}}
\put(150.0,131.0){\rule[-0.200pt]{0.400pt}{88.651pt}}
\put(150.0,131.0){\rule[-0.200pt]{165.980pt}{0.400pt}}
\put(839.0,131.0){\rule[-0.200pt]{0.400pt}{88.651pt}}
\put(150.0,499.0){\rule[-0.200pt]{165.980pt}{0.400pt}}
\end{picture}
\end{center}
\caption{The fitness value $f_t$  when  RSH-I maximises $(x-49)^2$.   $f_t$ is the mean fitness value   averaged over 100,000 runs.}
\label{f3-convergence1}
\end{figure}

In Figure~\ref{f3-convergence1}, RSH-I seems convergent, but this is only a kind of  premature convergence to the local optimum at $0$, rather than the global optimum at $100$.  In fact, the solution becomes farther away from the global optimum using the Euclidean distance: initially the   distance between the solution  and the optimum is $100-20=80$; then after $6000$ iterations, the distance increases to $100-0=100$.  

The approach of using the convergence in distribution provides a more accurate description   of converegnce than  that of visualising the fitness value does.

\section{Convergence Rate} 

\label{secConvergenceRate}
\subsection{Theoretical Study: Average Convergence Rate}

In this subsection, we define the average convergence rate of randomised search heuristics and then present lower and upper bounds on the average convergence rate.
The convergence rate  is  how fast  a randomised search heuristic converges to the optimal set per iteration. It is an  important measure of the performance of randomised search heuristics, but  less studied.

Since randomised search heuristics belong to iterative methods,
 we adopt   the average convergence rate, commonly used in iterative methods  \citep[Definition 3.1]{varga2009matrix}.

\begin{definition} 
Assume the probability of the initial population $\Phi_0$ in the non-optimal set is larger than 0.  The \emph{average  rate of convergence for $t$ iterations} is given by the following logarithmic reduction: 
\begin{align}
-\frac{1}{t} \ln \frac{\parallel    \mathbf{q}_t  \parallel_1}{\parallel    \mathbf{q}_0  \parallel_1}.  
\end{align}
\end{definition}

Since $P(\Phi_t \in S_{\non}) =\parallel \mathbf{q}_t \parallel_1$, then the  average convergence rate for $t$ iterations can be rewritten as
\begin{align}
-\frac{1}{t}\ln   \frac{P(\Phi_t \in S_{\non})}{P(\Phi_0 \in S_{\non})} .
\end{align} 

In the above definition, we don't consider the case of  the initial population  in the optimal set with probability 1. In this case the algorithm already converges and no need to discuss the convergence rate.

Notice that
\begin{align*}
\frac{\parallel \mathbf{q}_{t} \parallel_1}{\parallel \mathbf{q}_0 \parallel_1} =\frac{\parallel \mathbf{q}_{t} \parallel_1}{\parallel \mathbf{q}_{t-1} \parallel_1}
\frac{\parallel \mathbf{q}_{t-} \parallel_1}{\parallel \mathbf{q}_{t-2} \parallel_1} \cdots \frac{\parallel \mathbf{q}_{1} \parallel_1}{\parallel \mathbf{q}_0 \parallel_1}, 
 \end{align*}
 and the average convergence rate is equal to the   logarithmic mean  
 \begin{align*}
 \begin{array}{lll}
&-\frac{1}{t} \ln \frac{\parallel \mathbf{q}_{t} \parallel_1}{\parallel \mathbf{q}_0 \parallel_1}  \\=&-\frac{1}{t} \left( \ln \frac{\parallel \mathbf{q}_{t} \parallel_1}{\parallel \mathbf{q}_{t-1} \parallel_1}+\cdots +\ln \frac{\parallel \mathbf{q}_{1} \parallel_1}{\parallel \mathbf{q}_0 \parallel_1} \right) \\
=& -\frac{1}{t} \left( \ln \frac{P(\Phi_t \in S_{\non})}{P(\Phi_{t-1} \in S_{\non})}+\cdots +\ln \frac{P(\Phi_1 \in S_{\non})}{P(\Phi_0 \in S_{\non})}\right).
 \end{array}
\end{align*}
 
The last formula shows that the average convergence rate  is the average reduction factor of the probability of $\Phi_t$ in the non-optimal set per iteration in terms of  the   logarithmic mean.

In the following we estimate the lower bound and upper bound of the average convergence rate. The following theorem gives a lower bound on  the average convergence rate.

\begin{theorem} 
\label{theRateLowerBound}
If a randomised search heuristic is convergent, then the averaged convergence rate    is lower-bounded by
 \begin{align}
  &-\frac{1}{t} \ln \frac{\parallel    \mathbf{q}_t  \parallel_1}{\parallel    \mathbf{q}_0  \parallel_1}  \ge -\frac{1}{t} \ln \parallel    (\mathbf{Q}^T)^t   \parallel_1,
\\
& -   \lim_{t \to +\infty}    \frac{1}{t} \ln \frac{\parallel    \mathbf{q}_t  \parallel_1}{\parallel    \mathbf{q}_0  \parallel_1}    \ge  -\ln \rho(\mathbf{Q}).
\end{align}
\end{theorem}
\begin{proof}
 From the matrix iteration
$
 \mathbf{q}_t = (\mathbf{Q}^T)^t \mathbf{q}_0,
$
we get
\begin{align*}
 \frac{1}{t} \ln \frac{\parallel    \mathbf{q}_t  \parallel_1}{\parallel    \mathbf{q}_0  \parallel_1}  = \frac{1}{t} \ln \frac{\parallel     (\mathbf{Q}^T)^t \mathbf{q}_0\parallel_1}{\parallel    \mathbf{q}_0  \parallel_1}    \le   \frac{1}{t} \ln \parallel    (\mathbf{Q}^T)^t   \parallel_1.
\end{align*}
Then  the first conclusion is proven.

  Let   $t \to +\infty$ and apply  Gelfand's spectral radius formula, then 
 $$ \lim_{t \to +\infty} \frac{ \ln \parallel (\mathbf{Q}^T)^t\parallel_1}{t}=  \ln \mathbf{\rho(\mathbf{Q}^T)} =\ln \mathbf{\rho(\mathbf{Q})}.$$

Thus we prove that
$$
- \lim_{t \to +\infty} \frac{1}{t} \ln \frac{\parallel    \mathbf{q}_t  \parallel_1}{\parallel    \mathbf{q}_0  \parallel_1}  \ge -\ln \mathbf{\rho(\mathbf{Q})}.
$$
We proves the second conclusion.
\end{proof}

The following theorems gives an upper bound on the average convergence rate.
\begin{theorem} 
\label{theRateUpperBound}
If a randomised search heuristic is convergent, then the average convergence rate is upper-bounded by
\begin{align}
&- \frac{1}{t}\ln \frac{\parallel \mathbf{q}_t \parallel_1}{\parallel \mathbf{q}_0 \parallel_1}    
\le    -\frac{1}{t} \ln \left( \parallel ((\mathbf{Q}^{T})^{-1})^t \parallel_1 \right)^{-1},\\
 &- \lim_{t \to +\infty} \frac{1}{t} \ln  \frac{\parallel \mathbf{q}_t \parallel_1}{\parallel \mathbf{q}_0 \parallel_1}   \le   \ln {\rho(\mathbf{Q}^{-1})}
\end{align}
\end{theorem}

\begin{proof}
From $\mathbf{q}_t = (\mathbf{Q}^T)^t \mathbf{q}_0$, we get
\begin{align*}
&\mathbf{q}_0 = ((\mathbf{Q}^{T})^{-1})^t \mathbf{q}_t,
\end{align*}

Hence
\begin{align*}
&\parallel \mathbf{q}_0 \parallel_1 \le \parallel ((\mathbf{Q}^{T})^{-1})^t \parallel_1 \parallel \mathbf{q}_t \parallel_1,\\
&\frac{\parallel \mathbf{q}_0 \parallel_1}{\parallel \mathbf{q}_t \parallel_1} \le \parallel ((\mathbf{Q}^{T})^{-1})^t \parallel_1,
\end{align*}

then 
\begin{align*}
- \frac{1}{t}\ln \frac{\parallel \mathbf{q}_t \parallel_1}{\parallel \mathbf{q}_0 \parallel_1}    
\le &  -\frac{1}{t} \ln \left( \parallel ((\mathbf{Q}^{T})^{-1})^t \parallel_1 \right)^{-1}.
\end{align*}
Then the first conclusion is proven.

According to Gelfand's spectral radius formula  and the fact  $ \rho((\mathbf{Q}^{T})^{-1})=\rho(\mathbf{Q}^{-1})$, we get 
\begin{align*} 
 \lim_{t \to +\infty} \left(\parallel ((\mathbf{Q}^{T})^{-1})^t \parallel_1\right)^{1/t} =\rho((\mathbf{Q}^{T})^{-1})=\rho(\mathbf{Q}^{-1}).
\end{align*}
Then 
\begin{align*}
 -\lim_{t \to +\infty} \frac{1}{t}\ln \frac{\parallel \mathbf{q}_t \parallel_1}{\parallel \mathbf{q}_t \parallel_1}      
 \le &  \ln \rho( \mathbf{Q} ^{-1}).
\end{align*}
which is the second conclusion.
\end{proof}
 
From  the theoretical viewpoint, the above two theorems  show lower and upper  bounds on the average convergence rate. 
But in practice it is hard to apply the theoretical results since  both spectral radii $\rho( \mathbf{Q})$ and $\rho( \mathbf{Q} ^{-1})$ are too difficult to calculate in most cases.

\subsection{Computational Study} 
In this subsection, we propose a computational approach to illustrating  the average convergence rate of randomised heuristics. Unlike the theoretical study, the calculation of the average convergence rate is rather simple in the computational study.

We run a randomised search heuristic for $k$ times. Let $n  (\Phi_t \in S_{\non}) $ denote the number of $\Phi_t$ (where $t=0,1, \cdots$) appearing  in the non-optimal set for these $k$ runs. Then in practice, we will take 
\begin{align}
  -\frac{1}{t}\ln    \frac{n  (\Phi_t \in S_{\non})}{k}
\end{align}
as the average convergence rate. 

\begin{example}
\label{exaAverageRate1}
Consider  RSH-I and RSH-II for solving the  maximising problem 
$$
\max \, x^2, x \in \{ 0, \cdots, 100 \}.$$
\end{example}
 
 We run each algorithm for 100,000 times. The initial population is $\Phi_0=20$.   If    $n(\Phi_t \in S_{\non}) \le 10^{-5}$ happens, we don't calculate the average convergence rate.   It is due to the following reason: the event  $n(\Phi_t \in S_{\non}) \le 10^{-5}$  is a small probability event. Computer simulation in 100,000 runs is not enough from the statistical viewpoint.

 Figure \ref{f1-rate1c} shows  the average convergence rate  of RSH-I is much higher than that of RSA-II.  Initially the average convergence rate of both algorithms is $0$. Then the average convergence rate of RSH-I increases from $0$ to about $0.0009$ quickly,  but  the average rate of RSH-II increases from $0$ to about $0.0004$ slowly.

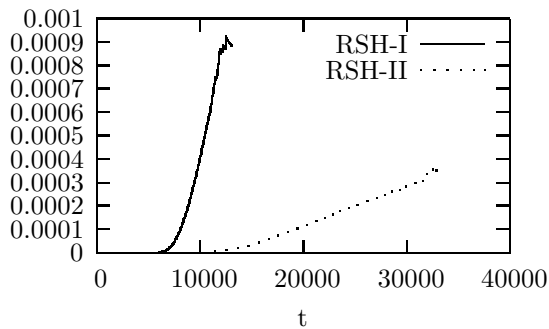
\begin{figure}[ht]
\begin{center}
\setlength{\unitlength}{0.240900pt}
\ifx\plotpoint\undefined\newsavebox{\plotpoint}\fi
\begin{picture}(900,540)(0,0)
\sbox{\plotpoint}{\rule[-0.200pt]{0.400pt}{0.400pt}}%
\put(190.0,131.0){\rule[-0.200pt]{4.818pt}{0.400pt}}
\put(170,131){\makebox(0,0)[r]{ 0}}
\put(819.0,131.0){\rule[-0.200pt]{4.818pt}{0.400pt}}
\put(190.0,168.0){\rule[-0.200pt]{4.818pt}{0.400pt}}
\put(170,168){\makebox(0,0)[r]{ 0.0001}}
\put(819.0,168.0){\rule[-0.200pt]{4.818pt}{0.400pt}}
\put(190.0,205.0){\rule[-0.200pt]{4.818pt}{0.400pt}}
\put(170,205){\makebox(0,0)[r]{ 0.0002}}
\put(819.0,205.0){\rule[-0.200pt]{4.818pt}{0.400pt}}
\put(190.0,241.0){\rule[-0.200pt]{4.818pt}{0.400pt}}
\put(170,241){\makebox(0,0)[r]{ 0.0003}}
\put(819.0,241.0){\rule[-0.200pt]{4.818pt}{0.400pt}}
\put(190.0,278.0){\rule[-0.200pt]{4.818pt}{0.400pt}}
\put(170,278){\makebox(0,0)[r]{ 0.0004}}
\put(819.0,278.0){\rule[-0.200pt]{4.818pt}{0.400pt}}
\put(190.0,315.0){\rule[-0.200pt]{4.818pt}{0.400pt}}
\put(170,315){\makebox(0,0)[r]{ 0.0005}}
\put(819.0,315.0){\rule[-0.200pt]{4.818pt}{0.400pt}}
\put(190.0,352.0){\rule[-0.200pt]{4.818pt}{0.400pt}}
\put(170,352){\makebox(0,0)[r]{ 0.0006}}
\put(819.0,352.0){\rule[-0.200pt]{4.818pt}{0.400pt}}
\put(190.0,389.0){\rule[-0.200pt]{4.818pt}{0.400pt}}
\put(170,389){\makebox(0,0)[r]{ 0.0007}}
\put(819.0,389.0){\rule[-0.200pt]{4.818pt}{0.400pt}}
\put(190.0,425.0){\rule[-0.200pt]{4.818pt}{0.400pt}}
\put(170,425){\makebox(0,0)[r]{ 0.0008}}
\put(819.0,425.0){\rule[-0.200pt]{4.818pt}{0.400pt}}
\put(190.0,462.0){\rule[-0.200pt]{4.818pt}{0.400pt}}
\put(170,462){\makebox(0,0)[r]{ 0.0009}}
\put(819.0,462.0){\rule[-0.200pt]{4.818pt}{0.400pt}}
\put(190.0,499.0){\rule[-0.200pt]{4.818pt}{0.400pt}}
\put(170,499){\makebox(0,0)[r]{ 0.001}}
\put(819.0,499.0){\rule[-0.200pt]{4.818pt}{0.400pt}}
\put(190.0,131.0){\rule[-0.200pt]{0.400pt}{4.818pt}}
\put(190,90){\makebox(0,0){ 0}}
\put(190.0,479.0){\rule[-0.200pt]{0.400pt}{4.818pt}}
\put(352.0,131.0){\rule[-0.200pt]{0.400pt}{4.818pt}}
\put(352,90){\makebox(0,0){ 10000}}
\put(352.0,479.0){\rule[-0.200pt]{0.400pt}{4.818pt}}
\put(514.0,131.0){\rule[-0.200pt]{0.400pt}{4.818pt}}
\put(514,90){\makebox(0,0){ 20000}}
\put(514.0,479.0){\rule[-0.200pt]{0.400pt}{4.818pt}}
\put(677.0,131.0){\rule[-0.200pt]{0.400pt}{4.818pt}}
\put(677,90){\makebox(0,0){ 30000}}
\put(677.0,479.0){\rule[-0.200pt]{0.400pt}{4.818pt}}
\put(839.0,131.0){\rule[-0.200pt]{0.400pt}{4.818pt}}
\put(839,90){\makebox(0,0){ 40000}}
\put(839.0,479.0){\rule[-0.200pt]{0.400pt}{4.818pt}}
\put(190.0,131.0){\rule[-0.200pt]{0.400pt}{88.651pt}}
\put(190.0,131.0){\rule[-0.200pt]{156.344pt}{0.400pt}}
\put(839.0,131.0){\rule[-0.200pt]{0.400pt}{88.651pt}}
\put(190.0,499.0){\rule[-0.200pt]{156.344pt}{0.400pt}}
\put(514,29){\makebox(0,0){t}}
\put(679,459){\makebox(0,0)[r]{RSH-I}}
\put(699.0,459.0){\rule[-0.200pt]{24.090pt}{0.400pt}}
\put(190,131){\usebox{\plotpoint}}
\put(190,131){\usebox{\plotpoint}}
\put(190,131){\usebox{\plotpoint}}
\put(190,131){\usebox{\plotpoint}}
\put(190,131){\usebox{\plotpoint}}
\put(190,131){\usebox{\plotpoint}}
\put(190,131){\usebox{\plotpoint}}
\put(190,131){\usebox{\plotpoint}}
\put(190,131){\usebox{\plotpoint}}
\put(190,131){\usebox{\plotpoint}}
\put(190,131){\usebox{\plotpoint}}
\put(190,131){\usebox{\plotpoint}}
\put(190,131){\usebox{\plotpoint}}
\put(190,131){\usebox{\plotpoint}}
\put(190,131){\usebox{\plotpoint}}
\put(190,131){\usebox{\plotpoint}}
\put(190,131){\usebox{\plotpoint}}
\put(190,131){\usebox{\plotpoint}}
\put(190,131){\usebox{\plotpoint}}
\put(190,131){\usebox{\plotpoint}}
\put(190,131){\usebox{\plotpoint}}
\put(190,131){\usebox{\plotpoint}}
\put(190,131){\usebox{\plotpoint}}
\put(190,131){\usebox{\plotpoint}}
\put(190,131){\usebox{\plotpoint}}
\put(190,131){\usebox{\plotpoint}}
\put(190,131){\usebox{\plotpoint}}
\put(190,131){\usebox{\plotpoint}}
\put(190,131){\usebox{\plotpoint}}
\put(190,131){\usebox{\plotpoint}}
\put(190,131){\usebox{\plotpoint}}
\put(190.0,131.0){\rule[-0.200pt]{23.608pt}{0.400pt}}
\put(288.0,131.0){\usebox{\plotpoint}}
\put(288.0,132.0){\rule[-0.200pt]{1.204pt}{0.400pt}}
\put(293.0,132.0){\usebox{\plotpoint}}
\put(293.0,133.0){\rule[-0.200pt]{0.723pt}{0.400pt}}
\put(296.0,133.0){\usebox{\plotpoint}}
\put(296.0,134.0){\rule[-0.200pt]{0.482pt}{0.400pt}}
\put(298.0,134.0){\usebox{\plotpoint}}
\put(298.0,135.0){\rule[-0.200pt]{0.482pt}{0.400pt}}
\put(300.0,135.0){\usebox{\plotpoint}}
\put(300.0,136.0){\usebox{\plotpoint}}
\put(301.0,136.0){\usebox{\plotpoint}}
\put(301.0,137.0){\rule[-0.200pt]{0.482pt}{0.400pt}}
\put(303.0,137.0){\usebox{\plotpoint}}
\put(303.0,138.0){\usebox{\plotpoint}}
\put(304.0,138.0){\usebox{\plotpoint}}
\put(304.0,139.0){\usebox{\plotpoint}}
\put(305.0,139.0){\usebox{\plotpoint}}
\put(305.0,140.0){\usebox{\plotpoint}}
\put(306.0,140.0){\usebox{\plotpoint}}
\put(306.0,141.0){\usebox{\plotpoint}}
\put(307.0,141.0){\usebox{\plotpoint}}
\put(307.0,142.0){\usebox{\plotpoint}}
\put(308.0,142.0){\rule[-0.200pt]{0.400pt}{0.482pt}}
\put(308.0,144.0){\usebox{\plotpoint}}
\put(309.0,144.0){\usebox{\plotpoint}}
\put(309.0,145.0){\usebox{\plotpoint}}
\put(310,145.67){\rule{0.241pt}{0.400pt}}
\multiput(310.00,145.17)(0.500,1.000){2}{\rule{0.120pt}{0.400pt}}
\put(310.0,145.0){\usebox{\plotpoint}}
\put(311,147){\usebox{\plotpoint}}
\put(311,147){\usebox{\plotpoint}}
\put(311,147){\usebox{\plotpoint}}
\put(311,147){\usebox{\plotpoint}}
\put(311,147){\usebox{\plotpoint}}
\put(311,147){\usebox{\plotpoint}}
\put(311,147){\usebox{\plotpoint}}
\put(311,147){\usebox{\plotpoint}}
\put(311,147){\usebox{\plotpoint}}
\put(311,147){\usebox{\plotpoint}}
\put(311,147){\usebox{\plotpoint}}
\put(311,147){\usebox{\plotpoint}}
\put(311,147){\usebox{\plotpoint}}
\put(311,147){\usebox{\plotpoint}}
\put(311,147){\usebox{\plotpoint}}
\put(311,147){\usebox{\plotpoint}}
\put(311,147){\usebox{\plotpoint}}
\put(311,147){\usebox{\plotpoint}}
\put(311,147){\usebox{\plotpoint}}
\put(311,147){\usebox{\plotpoint}}
\put(311,147){\usebox{\plotpoint}}
\put(311,147){\usebox{\plotpoint}}
\put(311,147){\usebox{\plotpoint}}
\put(311,147){\usebox{\plotpoint}}
\put(311,147){\usebox{\plotpoint}}
\put(311,147){\usebox{\plotpoint}}
\put(311,147){\usebox{\plotpoint}}
\put(311,147){\usebox{\plotpoint}}
\put(311,147){\usebox{\plotpoint}}
\put(311,147){\usebox{\plotpoint}}
\put(311,147){\usebox{\plotpoint}}
\put(311,147){\usebox{\plotpoint}}
\put(311,147){\usebox{\plotpoint}}
\put(311,147){\usebox{\plotpoint}}
\put(311,147){\usebox{\plotpoint}}
\put(311,147){\usebox{\plotpoint}}
\put(311,147){\usebox{\plotpoint}}
\put(311,147){\usebox{\plotpoint}}
\put(311.0,147.0){\usebox{\plotpoint}}
\put(311.0,148.0){\usebox{\plotpoint}}
\put(312.0,148.0){\rule[-0.200pt]{0.400pt}{0.482pt}}
\put(312.0,150.0){\usebox{\plotpoint}}
\put(313.0,150.0){\rule[-0.200pt]{0.400pt}{0.482pt}}
\put(313.0,152.0){\usebox{\plotpoint}}
\put(314.0,152.0){\usebox{\plotpoint}}
\put(314.0,153.0){\usebox{\plotpoint}}
\put(315.0,153.0){\rule[-0.200pt]{0.400pt}{0.482pt}}
\put(315.0,155.0){\usebox{\plotpoint}}
\put(316.0,155.0){\rule[-0.200pt]{0.400pt}{0.482pt}}
\put(316.0,157.0){\usebox{\plotpoint}}
\put(317.0,157.0){\rule[-0.200pt]{0.400pt}{0.723pt}}
\put(317.0,160.0){\usebox{\plotpoint}}
\put(318.0,160.0){\rule[-0.200pt]{0.400pt}{0.482pt}}
\put(318.0,162.0){\usebox{\plotpoint}}
\put(319.0,162.0){\rule[-0.200pt]{0.400pt}{0.482pt}}
\put(319.0,164.0){\usebox{\plotpoint}}
\put(320,165.67){\rule{0.241pt}{0.400pt}}
\multiput(320.00,165.17)(0.500,1.000){2}{\rule{0.120pt}{0.400pt}}
\put(320.0,164.0){\rule[-0.200pt]{0.400pt}{0.482pt}}
\put(321,167){\usebox{\plotpoint}}
\put(321,167){\usebox{\plotpoint}}
\put(321,167){\usebox{\plotpoint}}
\put(321,167){\usebox{\plotpoint}}
\put(321,167){\usebox{\plotpoint}}
\put(321,167){\usebox{\plotpoint}}
\put(321,167){\usebox{\plotpoint}}
\put(321,167){\usebox{\plotpoint}}
\put(321,167){\usebox{\plotpoint}}
\put(321,167){\usebox{\plotpoint}}
\put(321,167){\usebox{\plotpoint}}
\put(321,167){\usebox{\plotpoint}}
\put(321,167){\usebox{\plotpoint}}
\put(321,167){\usebox{\plotpoint}}
\put(321,167){\usebox{\plotpoint}}
\put(321,167){\usebox{\plotpoint}}
\put(321,167){\usebox{\plotpoint}}
\put(321,167){\usebox{\plotpoint}}
\put(321,167){\usebox{\plotpoint}}
\put(321,167){\usebox{\plotpoint}}
\put(321,167){\usebox{\plotpoint}}
\put(321,167){\usebox{\plotpoint}}
\put(321,167){\usebox{\plotpoint}}
\put(321,167){\usebox{\plotpoint}}
\put(321,167){\usebox{\plotpoint}}
\put(321.0,167.0){\rule[-0.200pt]{0.400pt}{0.482pt}}
\put(321.0,169.0){\usebox{\plotpoint}}
\put(322.0,169.0){\rule[-0.200pt]{0.400pt}{0.482pt}}
\put(322.0,171.0){\usebox{\plotpoint}}
\put(323.0,171.0){\rule[-0.200pt]{0.400pt}{0.723pt}}
\put(323.0,174.0){\usebox{\plotpoint}}
\put(324.0,174.0){\rule[-0.200pt]{0.400pt}{0.723pt}}
\put(324.0,177.0){\usebox{\plotpoint}}
\put(325.0,177.0){\rule[-0.200pt]{0.400pt}{0.723pt}}
\put(325.0,180.0){\usebox{\plotpoint}}
\put(326.0,180.0){\rule[-0.200pt]{0.400pt}{0.723pt}}
\put(326.0,183.0){\usebox{\plotpoint}}
\put(327.0,183.0){\rule[-0.200pt]{0.400pt}{0.723pt}}
\put(327.0,186.0){\usebox{\plotpoint}}
\put(328.0,186.0){\rule[-0.200pt]{0.400pt}{0.723pt}}
\put(328.0,189.0){\usebox{\plotpoint}}
\put(329.0,189.0){\rule[-0.200pt]{0.400pt}{0.723pt}}
\put(329.0,192.0){\usebox{\plotpoint}}
\put(330.0,192.0){\rule[-0.200pt]{0.400pt}{0.723pt}}
\put(330.0,195.0){\usebox{\plotpoint}}
\put(331.0,195.0){\rule[-0.200pt]{0.400pt}{0.723pt}}
\put(331.0,198.0){\usebox{\plotpoint}}
\put(332.0,198.0){\rule[-0.200pt]{0.400pt}{0.964pt}}
\put(332.0,202.0){\usebox{\plotpoint}}
\put(333.0,202.0){\rule[-0.200pt]{0.400pt}{0.723pt}}
\put(333.0,205.0){\usebox{\plotpoint}}
\put(334.0,205.0){\rule[-0.200pt]{0.400pt}{0.964pt}}
\put(334.0,209.0){\usebox{\plotpoint}}
\put(335.0,209.0){\rule[-0.200pt]{0.400pt}{0.964pt}}
\put(335.0,213.0){\usebox{\plotpoint}}
\put(336,215.67){\rule{0.241pt}{0.400pt}}
\multiput(336.00,215.17)(0.500,1.000){2}{\rule{0.120pt}{0.400pt}}
\put(336.0,213.0){\rule[-0.200pt]{0.400pt}{0.723pt}}
\put(337,217){\usebox{\plotpoint}}
\put(337,217){\usebox{\plotpoint}}
\put(337,217){\usebox{\plotpoint}}
\put(337,217){\usebox{\plotpoint}}
\put(337,217){\usebox{\plotpoint}}
\put(337,217){\usebox{\plotpoint}}
\put(337,217){\usebox{\plotpoint}}
\put(337,217){\usebox{\plotpoint}}
\put(337,217){\usebox{\plotpoint}}
\put(337,217){\usebox{\plotpoint}}
\put(337,217){\usebox{\plotpoint}}
\put(337,217){\usebox{\plotpoint}}
\put(337,217){\usebox{\plotpoint}}
\put(337,217){\usebox{\plotpoint}}
\put(337,217){\usebox{\plotpoint}}
\put(337,217){\usebox{\plotpoint}}
\put(337.0,217.0){\rule[-0.200pt]{0.400pt}{0.723pt}}
\put(337.0,220.0){\usebox{\plotpoint}}
\put(338.0,220.0){\rule[-0.200pt]{0.400pt}{0.964pt}}
\put(338.0,224.0){\usebox{\plotpoint}}
\put(339.0,224.0){\rule[-0.200pt]{0.400pt}{0.964pt}}
\put(339.0,228.0){\usebox{\plotpoint}}
\put(340.0,228.0){\rule[-0.200pt]{0.400pt}{0.964pt}}
\put(340.0,232.0){\usebox{\plotpoint}}
\put(341.0,232.0){\rule[-0.200pt]{0.400pt}{0.964pt}}
\put(341.0,236.0){\usebox{\plotpoint}}
\put(342.0,236.0){\rule[-0.200pt]{0.400pt}{0.964pt}}
\put(342.0,240.0){\usebox{\plotpoint}}
\put(343.0,240.0){\rule[-0.200pt]{0.400pt}{0.964pt}}
\put(343.0,244.0){\usebox{\plotpoint}}
\put(344.0,244.0){\rule[-0.200pt]{0.400pt}{0.964pt}}
\put(344.0,248.0){\usebox{\plotpoint}}
\put(345.0,248.0){\rule[-0.200pt]{0.400pt}{0.964pt}}
\put(345.0,252.0){\usebox{\plotpoint}}
\put(346.0,252.0){\rule[-0.200pt]{0.400pt}{1.204pt}}
\put(346.0,257.0){\usebox{\plotpoint}}
\put(347.0,257.0){\rule[-0.200pt]{0.400pt}{0.964pt}}
\put(347.0,261.0){\usebox{\plotpoint}}
\put(348.0,261.0){\rule[-0.200pt]{0.400pt}{0.964pt}}
\put(348.0,265.0){\usebox{\plotpoint}}
\put(349.0,265.0){\rule[-0.200pt]{0.400pt}{0.964pt}}
\put(349.0,269.0){\usebox{\plotpoint}}
\put(350.0,269.0){\rule[-0.200pt]{0.400pt}{0.964pt}}
\put(350.0,273.0){\usebox{\plotpoint}}
\put(351.0,273.0){\rule[-0.200pt]{0.400pt}{1.204pt}}
\put(351.0,278.0){\usebox{\plotpoint}}
\put(352.0,278.0){\rule[-0.200pt]{0.400pt}{1.204pt}}
\put(352.0,283.0){\usebox{\plotpoint}}
\put(353.0,283.0){\rule[-0.200pt]{0.400pt}{0.964pt}}
\put(353.0,287.0){\usebox{\plotpoint}}
\put(354.0,287.0){\rule[-0.200pt]{0.400pt}{1.445pt}}
\put(354.0,293.0){\usebox{\plotpoint}}
\put(355.0,293.0){\rule[-0.200pt]{0.400pt}{1.204pt}}
\put(355.0,298.0){\usebox{\plotpoint}}
\put(356.0,298.0){\rule[-0.200pt]{0.400pt}{0.964pt}}
\put(356.0,302.0){\usebox{\plotpoint}}
\put(357.0,302.0){\usebox{\plotpoint}}
\put(357.0,302.0){\usebox{\plotpoint}}
\put(357.0,302.0){\rule[-0.200pt]{0.400pt}{1.204pt}}
\put(357.0,307.0){\usebox{\plotpoint}}
\put(358.0,307.0){\rule[-0.200pt]{0.400pt}{0.964pt}}
\put(358.0,310.0){\usebox{\plotpoint}}
\put(358.0,310.0){\rule[-0.200pt]{0.400pt}{0.482pt}}
\put(358.0,312.0){\usebox{\plotpoint}}
\put(359.0,312.0){\rule[-0.200pt]{0.400pt}{0.964pt}}
\put(359.0,316.0){\usebox{\plotpoint}}
\put(360.0,316.0){\rule[-0.200pt]{0.400pt}{0.964pt}}
\put(360.0,320.0){\usebox{\plotpoint}}
\put(361.0,320.0){\usebox{\plotpoint}}
\put(361.0,320.0){\usebox{\plotpoint}}
\put(361.0,320.0){\rule[-0.200pt]{0.400pt}{1.204pt}}
\put(361.0,325.0){\usebox{\plotpoint}}
\put(362.0,325.0){\rule[-0.200pt]{0.400pt}{1.204pt}}
\put(362.0,330.0){\usebox{\plotpoint}}
\put(363.0,330.0){\rule[-0.200pt]{0.400pt}{0.723pt}}
\put(363.0,332.0){\usebox{\plotpoint}}
\put(363.0,332.0){\rule[-0.200pt]{0.400pt}{0.482pt}}
\put(363.0,333.0){\usebox{\plotpoint}}
\put(363,333.67){\rule{0.241pt}{0.400pt}}
\multiput(363.00,333.17)(0.500,1.000){2}{\rule{0.120pt}{0.400pt}}
\put(363.0,333.0){\usebox{\plotpoint}}
\put(364.0,334.0){\usebox{\plotpoint}}
\put(364.0,334.0){\rule[-0.200pt]{0.400pt}{1.204pt}}
\put(364.0,339.0){\usebox{\plotpoint}}
\put(365.0,339.0){\rule[-0.200pt]{0.400pt}{1.204pt}}
\put(365.0,344.0){\usebox{\plotpoint}}
\put(366.0,344.0){\usebox{\plotpoint}}
\put(366.0,344.0){\usebox{\plotpoint}}
\put(366.0,344.0){\rule[-0.200pt]{0.400pt}{0.964pt}}
\put(366.0,348.0){\usebox{\plotpoint}}
\put(367,352.67){\rule{0.241pt}{0.400pt}}
\multiput(367.00,353.17)(0.500,-1.000){2}{\rule{0.120pt}{0.400pt}}
\put(367.0,348.0){\rule[-0.200pt]{0.400pt}{1.445pt}}
\put(368,353){\usebox{\plotpoint}}
\put(368.0,353.0){\rule[-0.200pt]{0.400pt}{1.204pt}}
\put(368.0,357.0){\usebox{\plotpoint}}
\put(368.0,357.0){\rule[-0.200pt]{0.400pt}{0.723pt}}
\put(368.0,360.0){\usebox{\plotpoint}}
\put(369.0,360.0){\rule[-0.200pt]{0.400pt}{0.482pt}}
\put(369.0,361.0){\usebox{\plotpoint}}
\put(369,363.67){\rule{0.241pt}{0.400pt}}
\multiput(369.00,363.17)(0.500,1.000){2}{\rule{0.120pt}{0.400pt}}
\put(369.0,361.0){\rule[-0.200pt]{0.400pt}{0.723pt}}
\put(370,365){\usebox{\plotpoint}}
\put(370,365){\usebox{\plotpoint}}
\put(370,365){\usebox{\plotpoint}}
\put(370,365){\usebox{\plotpoint}}
\put(370,365){\usebox{\plotpoint}}
\put(370,365){\usebox{\plotpoint}}
\put(370,365){\usebox{\plotpoint}}
\put(370,365){\usebox{\plotpoint}}
\put(370,365){\usebox{\plotpoint}}
\put(370,365){\usebox{\plotpoint}}
\put(370,365){\usebox{\plotpoint}}
\put(370.0,365.0){\usebox{\plotpoint}}
\put(370.0,365.0){\usebox{\plotpoint}}
\put(370.0,365.0){\rule[-0.200pt]{0.400pt}{1.445pt}}
\put(370.0,370.0){\usebox{\plotpoint}}
\put(370.0,370.0){\rule[-0.200pt]{0.400pt}{0.482pt}}
\put(370.0,372.0){\usebox{\plotpoint}}
\put(371,375.67){\rule{0.241pt}{0.400pt}}
\multiput(371.00,375.17)(0.500,1.000){2}{\rule{0.120pt}{0.400pt}}
\put(371.0,372.0){\rule[-0.200pt]{0.400pt}{0.964pt}}
\put(372,377){\usebox{\plotpoint}}
\put(372,377){\usebox{\plotpoint}}
\put(372,377){\usebox{\plotpoint}}
\put(372,377){\usebox{\plotpoint}}
\put(372,377){\usebox{\plotpoint}}
\put(372,377){\usebox{\plotpoint}}
\put(372,377){\usebox{\plotpoint}}
\put(372,377){\usebox{\plotpoint}}
\put(372,377){\usebox{\plotpoint}}
\put(372.0,377.0){\usebox{\plotpoint}}
\put(372.0,377.0){\usebox{\plotpoint}}
\put(372.0,377.0){\rule[-0.200pt]{0.400pt}{0.482pt}}
\put(372.0,378.0){\usebox{\plotpoint}}
\put(372.0,378.0){\rule[-0.200pt]{0.400pt}{1.204pt}}
\put(372.0,382.0){\usebox{\plotpoint}}
\put(372.0,382.0){\usebox{\plotpoint}}
\put(373.0,382.0){\rule[-0.200pt]{0.400pt}{1.927pt}}
\put(373.0,390.0){\usebox{\plotpoint}}
\put(374.0,390.0){\rule[-0.200pt]{0.400pt}{0.482pt}}
\put(374.0,391.0){\usebox{\plotpoint}}
\put(374.0,391.0){\rule[-0.200pt]{0.400pt}{1.204pt}}
\put(374.0,396.0){\usebox{\plotpoint}}
\put(375.0,396.0){\usebox{\plotpoint}}
\put(375.0,396.0){\usebox{\plotpoint}}
\put(375.0,396.0){\rule[-0.200pt]{0.400pt}{0.964pt}}
\put(375.0,399.0){\usebox{\plotpoint}}
\put(375.0,399.0){\rule[-0.200pt]{0.400pt}{1.204pt}}
\put(375.0,404.0){\usebox{\plotpoint}}
\put(376.0,403.0){\usebox{\plotpoint}}
\put(376.0,403.0){\rule[-0.200pt]{0.400pt}{1.204pt}}
\put(376.0,407.0){\usebox{\plotpoint}}
\put(376.0,407.0){\usebox{\plotpoint}}
\put(377.0,406.0){\usebox{\plotpoint}}
\put(377.0,406.0){\usebox{\plotpoint}}
\put(378.0,405.0){\usebox{\plotpoint}}
\put(378.0,405.0){\rule[-0.200pt]{0.400pt}{0.964pt}}
\put(378.0,408.0){\usebox{\plotpoint}}
\put(378.0,408.0){\rule[-0.200pt]{0.400pt}{0.482pt}}
\put(378.0,410.0){\usebox{\plotpoint}}
\put(379.0,410.0){\rule[-0.200pt]{0.400pt}{0.482pt}}
\put(379.0,411.0){\usebox{\plotpoint}}
\put(379.0,411.0){\rule[-0.200pt]{0.400pt}{1.927pt}}
\put(379.0,418.0){\usebox{\plotpoint}}
\put(379.0,418.0){\usebox{\plotpoint}}
\put(380.0,417.0){\usebox{\plotpoint}}
\put(380.0,417.0){\rule[-0.200pt]{0.400pt}{2.409pt}}
\put(380.0,427.0){\usebox{\plotpoint}}
\put(381.0,427.0){\rule[-0.200pt]{0.400pt}{1.927pt}}
\put(381.0,434.0){\usebox{\plotpoint}}
\put(381.0,434.0){\rule[-0.200pt]{0.400pt}{1.445pt}}
\put(381.0,440.0){\usebox{\plotpoint}}
\put(382.0,440.0){\rule[-0.200pt]{0.400pt}{1.686pt}}
\put(382.0,445.0){\rule[-0.200pt]{0.400pt}{0.482pt}}
\put(382.0,445.0){\usebox{\plotpoint}}
\put(383.0,443.0){\rule[-0.200pt]{0.400pt}{0.482pt}}
\put(383.0,443.0){\usebox{\plotpoint}}
\put(384.0,443.0){\rule[-0.200pt]{0.400pt}{1.927pt}}
\put(384.0,451.0){\usebox{\plotpoint}}
\put(385.0,449.0){\rule[-0.200pt]{0.400pt}{0.482pt}}
\put(385.0,449.0){\usebox{\plotpoint}}
\put(386.0,447.0){\rule[-0.200pt]{0.400pt}{0.482pt}}
\put(386.0,447.0){\usebox{\plotpoint}}
\put(387.0,446.0){\usebox{\plotpoint}}
\put(387.0,446.0){\usebox{\plotpoint}}
\put(388.0,446.0){\rule[-0.200pt]{0.400pt}{2.891pt}}
\put(388.0,456.0){\rule[-0.200pt]{0.400pt}{0.482pt}}
\put(388.0,456.0){\usebox{\plotpoint}}
\put(389.0,455.0){\usebox{\plotpoint}}
\put(389.0,455.0){\usebox{\plotpoint}}
\put(390.0,453.0){\rule[-0.200pt]{0.400pt}{0.482pt}}
\put(390.0,453.0){\usebox{\plotpoint}}
\put(391.0,452.0){\usebox{\plotpoint}}
\put(391.0,452.0){\usebox{\plotpoint}}
\put(392.0,451.0){\usebox{\plotpoint}}
\put(392.0,451.0){\rule[-0.200pt]{0.400pt}{5.059pt}}
\put(392.0,470.0){\rule[-0.200pt]{0.400pt}{0.482pt}}
\put(392.0,470.0){\usebox{\plotpoint}}
\put(393.0,469.0){\usebox{\plotpoint}}
\put(393.0,469.0){\usebox{\plotpoint}}
\put(394.0,467.0){\rule[-0.200pt]{0.400pt}{0.482pt}}
\put(394.0,467.0){\usebox{\plotpoint}}
\put(395.0,465.0){\rule[-0.200pt]{0.400pt}{0.482pt}}
\put(395.0,465.0){\usebox{\plotpoint}}
\put(396.0,464.0){\usebox{\plotpoint}}
\put(396.0,464.0){\usebox{\plotpoint}}
\put(397.0,462.0){\rule[-0.200pt]{0.400pt}{0.482pt}}
\put(397.0,462.0){\usebox{\plotpoint}}
\put(398.0,461.0){\usebox{\plotpoint}}
\put(398.0,461.0){\usebox{\plotpoint}}
\put(399.0,459.0){\rule[-0.200pt]{0.400pt}{0.482pt}}
\put(399.0,459.0){\usebox{\plotpoint}}
\put(400.0,458.0){\usebox{\plotpoint}}
\put(400.0,458.0){\usebox{\plotpoint}}
\put(401.0,456.0){\rule[-0.200pt]{0.400pt}{0.482pt}}
\put(401.0,456.0){\usebox{\plotpoint}}
\put(679,418){\makebox(0,0)[r]{RSH-II}}
\multiput(699,418)(20.756,0.000){5}{\usebox{\plotpoint}}
\put(799,418){\usebox{\plotpoint}}
\put(190,131){\usebox{\plotpoint}}
\put(190.00,131.00){\usebox{\plotpoint}}
\put(210.76,131.00){\usebox{\plotpoint}}
\put(231.51,131.00){\usebox{\plotpoint}}
\put(252.27,131.00){\usebox{\plotpoint}}
\put(273.02,131.00){\usebox{\plotpoint}}
\put(293.78,131.00){\usebox{\plotpoint}}
\put(314.53,131.00){\usebox{\plotpoint}}
\put(335.29,131.00){\usebox{\plotpoint}}
\put(356.04,131.00){\usebox{\plotpoint}}
\put(374.80,133.00){\usebox{\plotpoint}}
\put(393.56,135.00){\usebox{\plotpoint}}
\put(411.31,138.00){\usebox{\plotpoint}}
\put(428.07,142.00){\usebox{\plotpoint}}
\put(443.82,147.00){\usebox{\plotpoint}}
\put(459.58,152.00){\usebox{\plotpoint}}
\put(474.33,158.00){\usebox{\plotpoint}}
\put(490.00,163.09){\usebox{\plotpoint}}
\put(504.84,169.00){\usebox{\plotpoint}}
\put(519.60,175.00){\usebox{\plotpoint}}
\put(534.35,181.00){\usebox{\plotpoint}}
\put(548.11,186.00){\usebox{\plotpoint}}
\put(559.87,191.00){\usebox{\plotpoint}}
\put(571.62,196.00){\usebox{\plotpoint}}
\put(585.38,201.00){\usebox{\plotpoint}}
\put(599.13,206.00){\usebox{\plotpoint}}
\put(614.47,212.00){\usebox{\plotpoint}}
\put(630.23,217.00){\usebox{\plotpoint}}
\put(642.00,221.98){\usebox{\plotpoint}}
\put(653.74,227.00){\usebox{\plotpoint}}
\put(665.50,230.00){\usebox{\plotpoint}}
\put(675.25,235.00){\usebox{\plotpoint}}
\put(690.01,241.00){\usebox{\plotpoint}}
\put(701.76,244.00){\usebox{\plotpoint}}
\put(708.00,254.52){\usebox{\plotpoint}}
\put(717.00,261.73){\usebox{\plotpoint}}
\put(723,260){\usebox{\plotpoint}}
\put(190.0,131.0){\rule[-0.200pt]{0.400pt}{88.651pt}}
\put(190.0,131.0){\rule[-0.200pt]{156.344pt}{0.400pt}}
\put(839.0,131.0){\rule[-0.200pt]{0.400pt}{88.651pt}}
\put(190.0,499.0){\rule[-0.200pt]{156.344pt}{0.400pt}}
\end{picture}
\end{center}
\caption{The average convergence rate of RSH-I and RSH-II for maximising $x^2$.}
\label{f1-rate1c}
\end{figure}

The average convergence rate is different from   the  logarithmic progress rate,  $\ln f_t$, used in some references (for example, see Figures 8 and 9 in~\cite{salomon1998evolutionary}).   Such a logarithmic rate may provide an intuitive description of the  fitness change, but does not give a quantitative measure of the convergence rate itself. Let's demonstrate this by the following example.
 
\begin{example}
Consider RSH-II  for solving the following two   problems,   
$$\max \, x^2, \quad \max \,=10x^2, \quad x \in \{ 0,1, \cdots, 100 \}.
$$  

We run the algorithm for 100,000 times. The initial population is $\Phi_0=20$.  
Figure  \ref{f1-f3} depicts  that the logarithmic progress rate of RSH-II  on the function $10 x^2$ is lager than that on the function $x^2$. It is caused by the coefficient difference between  the two fitnesses.   The logarithmic progress rate is not a quantitative  measure  of the convergence rate.

\begin{figure}[ht]
\begin{center}
\setlength{\unitlength}{0.240900pt}
\ifx\plotpoint\undefined\newsavebox{\plotpoint}\fi
\begin{picture}(900,540)(0,0)
\sbox{\plotpoint}{\rule[-0.200pt]{0.400pt}{0.400pt}}%
\put(110.0,131.0){\rule[-0.200pt]{4.818pt}{0.400pt}}
\put(90,131){\makebox(0,0)[r]{ 5}}
\put(819.0,131.0){\rule[-0.200pt]{4.818pt}{0.400pt}}
\put(110.0,184.0){\rule[-0.200pt]{4.818pt}{0.400pt}}
\put(90,184){\makebox(0,0)[r]{ 6}}
\put(819.0,184.0){\rule[-0.200pt]{4.818pt}{0.400pt}}
\put(110.0,236.0){\rule[-0.200pt]{4.818pt}{0.400pt}}
\put(90,236){\makebox(0,0)[r]{ 7}}
\put(819.0,236.0){\rule[-0.200pt]{4.818pt}{0.400pt}}
\put(110.0,289.0){\rule[-0.200pt]{4.818pt}{0.400pt}}
\put(90,289){\makebox(0,0)[r]{ 8}}
\put(819.0,289.0){\rule[-0.200pt]{4.818pt}{0.400pt}}
\put(110.0,341.0){\rule[-0.200pt]{4.818pt}{0.400pt}}
\put(90,341){\makebox(0,0)[r]{ 9}}
\put(819.0,341.0){\rule[-0.200pt]{4.818pt}{0.400pt}}
\put(110.0,394.0){\rule[-0.200pt]{4.818pt}{0.400pt}}
\put(90,394){\makebox(0,0)[r]{ 10}}
\put(819.0,394.0){\rule[-0.200pt]{4.818pt}{0.400pt}}
\put(110.0,446.0){\rule[-0.200pt]{4.818pt}{0.400pt}}
\put(90,446){\makebox(0,0)[r]{ 11}}
\put(819.0,446.0){\rule[-0.200pt]{4.818pt}{0.400pt}}
\put(110.0,499.0){\rule[-0.200pt]{4.818pt}{0.400pt}}
\put(90,499){\makebox(0,0)[r]{ 12}}
\put(819.0,499.0){\rule[-0.200pt]{4.818pt}{0.400pt}}
\put(110.0,131.0){\rule[-0.200pt]{0.400pt}{4.818pt}}
\put(110,90){\makebox(0,0){ 0}}
\put(110.0,479.0){\rule[-0.200pt]{0.400pt}{4.818pt}}
\put(256.0,131.0){\rule[-0.200pt]{0.400pt}{4.818pt}}
\put(256,90){\makebox(0,0){ 4000}}
\put(256.0,479.0){\rule[-0.200pt]{0.400pt}{4.818pt}}
\put(402.0,131.0){\rule[-0.200pt]{0.400pt}{4.818pt}}
\put(402,90){\makebox(0,0){ 8000}}
\put(402.0,479.0){\rule[-0.200pt]{0.400pt}{4.818pt}}
\put(547.0,131.0){\rule[-0.200pt]{0.400pt}{4.818pt}}
\put(547,90){\makebox(0,0){ 12000}}
\put(547.0,479.0){\rule[-0.200pt]{0.400pt}{4.818pt}}
\put(693.0,131.0){\rule[-0.200pt]{0.400pt}{4.818pt}}
\put(693,90){\makebox(0,0){ 16000}}
\put(693.0,479.0){\rule[-0.200pt]{0.400pt}{4.818pt}}
\put(839.0,131.0){\rule[-0.200pt]{0.400pt}{4.818pt}}
\put(839,90){\makebox(0,0){ 20000}}
\put(839.0,479.0){\rule[-0.200pt]{0.400pt}{4.818pt}}
\put(110.0,131.0){\rule[-0.200pt]{0.400pt}{88.651pt}}
\put(110.0,131.0){\rule[-0.200pt]{175.616pt}{0.400pt}}
\put(839.0,131.0){\rule[-0.200pt]{0.400pt}{88.651pt}}
\put(110.0,499.0){\rule[-0.200pt]{175.616pt}{0.400pt}}
\put(474,29){\makebox(0,0){t}}
\put(679,459){\makebox(0,0)[r]{$x^2$}}
\put(699.0,459.0){\rule[-0.200pt]{24.090pt}{0.400pt}}
\put(110,183){\usebox{\plotpoint}}
\put(110,183){\usebox{\plotpoint}}
\put(110,183){\usebox{\plotpoint}}
\put(110,183){\usebox{\plotpoint}}
\put(110,183){\usebox{\plotpoint}}
\put(110,183){\usebox{\plotpoint}}
\put(110,183){\usebox{\plotpoint}}
\put(110,183){\usebox{\plotpoint}}
\put(110.0,183.0){\usebox{\plotpoint}}
\put(110.0,184.0){\usebox{\plotpoint}}
\put(111.0,184.0){\usebox{\plotpoint}}
\put(111.0,185.0){\usebox{\plotpoint}}
\put(112.0,185.0){\rule[-0.200pt]{0.400pt}{0.482pt}}
\put(112.0,187.0){\usebox{\plotpoint}}
\put(113.0,187.0){\usebox{\plotpoint}}
\put(113.0,188.0){\usebox{\plotpoint}}
\put(114.0,188.0){\rule[-0.200pt]{0.400pt}{0.482pt}}
\put(114.0,190.0){\usebox{\plotpoint}}
\put(115.0,190.0){\usebox{\plotpoint}}
\put(115.0,191.0){\usebox{\plotpoint}}
\put(116.0,191.0){\usebox{\plotpoint}}
\put(116.0,192.0){\usebox{\plotpoint}}
\put(117.0,192.0){\rule[-0.200pt]{0.400pt}{0.482pt}}
\put(117.0,194.0){\usebox{\plotpoint}}
\put(118.0,194.0){\usebox{\plotpoint}}
\put(118.0,195.0){\usebox{\plotpoint}}
\put(119.0,195.0){\usebox{\plotpoint}}
\put(119.0,196.0){\usebox{\plotpoint}}
\put(120.0,196.0){\rule[-0.200pt]{0.400pt}{0.482pt}}
\put(120.0,198.0){\usebox{\plotpoint}}
\put(121.0,198.0){\usebox{\plotpoint}}
\put(121.0,199.0){\usebox{\plotpoint}}
\put(122.0,199.0){\usebox{\plotpoint}}
\put(122.0,200.0){\usebox{\plotpoint}}
\put(123.0,200.0){\usebox{\plotpoint}}
\put(123.0,201.0){\usebox{\plotpoint}}
\put(124.0,201.0){\rule[-0.200pt]{0.400pt}{0.482pt}}
\put(124.0,203.0){\usebox{\plotpoint}}
\put(125.0,203.0){\usebox{\plotpoint}}
\put(125.0,204.0){\usebox{\plotpoint}}
\put(126.0,204.0){\usebox{\plotpoint}}
\put(126.0,205.0){\usebox{\plotpoint}}
\put(127.0,205.0){\usebox{\plotpoint}}
\put(127.0,206.0){\usebox{\plotpoint}}
\put(128.0,206.0){\usebox{\plotpoint}}
\put(128.0,207.0){\usebox{\plotpoint}}
\put(129,207.67){\rule{0.241pt}{0.400pt}}
\multiput(129.00,207.17)(0.500,1.000){2}{\rule{0.120pt}{0.400pt}}
\put(129.0,207.0){\usebox{\plotpoint}}
\put(130,209){\usebox{\plotpoint}}
\put(130,209){\usebox{\plotpoint}}
\put(130,209){\usebox{\plotpoint}}
\put(130,209){\usebox{\plotpoint}}
\put(130,209){\usebox{\plotpoint}}
\put(130,209){\usebox{\plotpoint}}
\put(130,209){\usebox{\plotpoint}}
\put(130,209){\usebox{\plotpoint}}
\put(130,209){\usebox{\plotpoint}}
\put(130,209){\usebox{\plotpoint}}
\put(130,209){\usebox{\plotpoint}}
\put(130,209){\usebox{\plotpoint}}
\put(130,209){\usebox{\plotpoint}}
\put(130,209){\usebox{\plotpoint}}
\put(130,209){\usebox{\plotpoint}}
\put(130,209){\usebox{\plotpoint}}
\put(130,209){\usebox{\plotpoint}}
\put(130,209){\usebox{\plotpoint}}
\put(130,209){\usebox{\plotpoint}}
\put(130,209){\usebox{\plotpoint}}
\put(130,209){\usebox{\plotpoint}}
\put(130,209){\usebox{\plotpoint}}
\put(130,209){\usebox{\plotpoint}}
\put(130.0,209.0){\usebox{\plotpoint}}
\put(130.0,210.0){\usebox{\plotpoint}}
\put(131.0,210.0){\usebox{\plotpoint}}
\put(131.0,211.0){\usebox{\plotpoint}}
\put(132.0,211.0){\usebox{\plotpoint}}
\put(132.0,212.0){\usebox{\plotpoint}}
\put(133.0,212.0){\usebox{\plotpoint}}
\put(133.0,213.0){\usebox{\plotpoint}}
\put(134.0,213.0){\usebox{\plotpoint}}
\put(134.0,214.0){\usebox{\plotpoint}}
\put(135.0,214.0){\usebox{\plotpoint}}
\put(135.0,215.0){\usebox{\plotpoint}}
\put(136.0,215.0){\usebox{\plotpoint}}
\put(136.0,216.0){\usebox{\plotpoint}}
\put(137.0,216.0){\usebox{\plotpoint}}
\put(137.0,217.0){\usebox{\plotpoint}}
\put(138.0,217.0){\usebox{\plotpoint}}
\put(138.0,218.0){\usebox{\plotpoint}}
\put(139.0,218.0){\usebox{\plotpoint}}
\put(139.0,219.0){\usebox{\plotpoint}}
\put(140.0,219.0){\usebox{\plotpoint}}
\put(140.0,220.0){\usebox{\plotpoint}}
\put(141,220.67){\rule{0.241pt}{0.400pt}}
\multiput(141.00,220.17)(0.500,1.000){2}{\rule{0.120pt}{0.400pt}}
\put(141.0,220.0){\usebox{\plotpoint}}
\put(142,222){\usebox{\plotpoint}}
\put(142,222){\usebox{\plotpoint}}
\put(142,222){\usebox{\plotpoint}}
\put(142,222){\usebox{\plotpoint}}
\put(142,222){\usebox{\plotpoint}}
\put(142,222){\usebox{\plotpoint}}
\put(142,222){\usebox{\plotpoint}}
\put(142,222){\usebox{\plotpoint}}
\put(142,222){\usebox{\plotpoint}}
\put(142,222){\usebox{\plotpoint}}
\put(142,222){\usebox{\plotpoint}}
\put(142,222){\usebox{\plotpoint}}
\put(142,222){\usebox{\plotpoint}}
\put(142,222){\usebox{\plotpoint}}
\put(142,222){\usebox{\plotpoint}}
\put(142,222){\usebox{\plotpoint}}
\put(142,222){\usebox{\plotpoint}}
\put(142,222){\usebox{\plotpoint}}
\put(142,222){\usebox{\plotpoint}}
\put(142,222){\usebox{\plotpoint}}
\put(142,222){\usebox{\plotpoint}}
\put(142,222){\usebox{\plotpoint}}
\put(142,222){\usebox{\plotpoint}}
\put(142,222){\usebox{\plotpoint}}
\put(142,222){\usebox{\plotpoint}}
\put(142,222){\usebox{\plotpoint}}
\put(142,221.67){\rule{0.241pt}{0.400pt}}
\multiput(142.00,221.17)(0.500,1.000){2}{\rule{0.120pt}{0.400pt}}
\put(143,223){\usebox{\plotpoint}}
\put(143,223){\usebox{\plotpoint}}
\put(143,223){\usebox{\plotpoint}}
\put(143,223){\usebox{\plotpoint}}
\put(143,223){\usebox{\plotpoint}}
\put(143,223){\usebox{\plotpoint}}
\put(143,223){\usebox{\plotpoint}}
\put(143,223){\usebox{\plotpoint}}
\put(143,223){\usebox{\plotpoint}}
\put(143,223){\usebox{\plotpoint}}
\put(143,223){\usebox{\plotpoint}}
\put(143,223){\usebox{\plotpoint}}
\put(143,223){\usebox{\plotpoint}}
\put(143,223){\usebox{\plotpoint}}
\put(143,223){\usebox{\plotpoint}}
\put(143,223){\usebox{\plotpoint}}
\put(143,223){\usebox{\plotpoint}}
\put(143,223){\usebox{\plotpoint}}
\put(143,223){\usebox{\plotpoint}}
\put(143,223){\usebox{\plotpoint}}
\put(143,223){\usebox{\plotpoint}}
\put(143,223){\usebox{\plotpoint}}
\put(143,223){\usebox{\plotpoint}}
\put(143,223){\usebox{\plotpoint}}
\put(143,223){\usebox{\plotpoint}}
\put(143,223){\usebox{\plotpoint}}
\put(143,223){\usebox{\plotpoint}}
\put(143,222.67){\rule{0.241pt}{0.400pt}}
\multiput(143.00,222.17)(0.500,1.000){2}{\rule{0.120pt}{0.400pt}}
\put(144,224){\usebox{\plotpoint}}
\put(144,224){\usebox{\plotpoint}}
\put(144,224){\usebox{\plotpoint}}
\put(144,224){\usebox{\plotpoint}}
\put(144,224){\usebox{\plotpoint}}
\put(144,224){\usebox{\plotpoint}}
\put(144,224){\usebox{\plotpoint}}
\put(144,224){\usebox{\plotpoint}}
\put(144,224){\usebox{\plotpoint}}
\put(144,224){\usebox{\plotpoint}}
\put(144,224){\usebox{\plotpoint}}
\put(144,224){\usebox{\plotpoint}}
\put(144,224){\usebox{\plotpoint}}
\put(144,224){\usebox{\plotpoint}}
\put(144,224){\usebox{\plotpoint}}
\put(144,224){\usebox{\plotpoint}}
\put(144,224){\usebox{\plotpoint}}
\put(144,224){\usebox{\plotpoint}}
\put(144,224){\usebox{\plotpoint}}
\put(144,224){\usebox{\plotpoint}}
\put(144,224){\usebox{\plotpoint}}
\put(144,224){\usebox{\plotpoint}}
\put(144,224){\usebox{\plotpoint}}
\put(144,224){\usebox{\plotpoint}}
\put(144,224){\usebox{\plotpoint}}
\put(144,224){\usebox{\plotpoint}}
\put(144,223.67){\rule{0.241pt}{0.400pt}}
\multiput(144.00,223.17)(0.500,1.000){2}{\rule{0.120pt}{0.400pt}}
\put(145,225){\usebox{\plotpoint}}
\put(145,225){\usebox{\plotpoint}}
\put(145,225){\usebox{\plotpoint}}
\put(145,225){\usebox{\plotpoint}}
\put(145,225){\usebox{\plotpoint}}
\put(145,225){\usebox{\plotpoint}}
\put(145,225){\usebox{\plotpoint}}
\put(145,225){\usebox{\plotpoint}}
\put(145,225){\usebox{\plotpoint}}
\put(145,225){\usebox{\plotpoint}}
\put(145,225){\usebox{\plotpoint}}
\put(145,225){\usebox{\plotpoint}}
\put(145,225){\usebox{\plotpoint}}
\put(145,225){\usebox{\plotpoint}}
\put(145,225){\usebox{\plotpoint}}
\put(145,225){\usebox{\plotpoint}}
\put(145,225){\usebox{\plotpoint}}
\put(145,225){\usebox{\plotpoint}}
\put(145,225){\usebox{\plotpoint}}
\put(145,225){\usebox{\plotpoint}}
\put(145,225){\usebox{\plotpoint}}
\put(145,225){\usebox{\plotpoint}}
\put(145,225){\usebox{\plotpoint}}
\put(145,225){\usebox{\plotpoint}}
\put(145,225){\usebox{\plotpoint}}
\put(145,225){\usebox{\plotpoint}}
\put(145.0,225.0){\usebox{\plotpoint}}
\put(146.0,225.0){\usebox{\plotpoint}}
\put(146.0,226.0){\usebox{\plotpoint}}
\put(147.0,226.0){\usebox{\plotpoint}}
\put(147.0,227.0){\usebox{\plotpoint}}
\put(148.0,227.0){\usebox{\plotpoint}}
\put(148.0,228.0){\usebox{\plotpoint}}
\put(149.0,228.0){\usebox{\plotpoint}}
\put(149.0,229.0){\usebox{\plotpoint}}
\put(150.0,229.0){\usebox{\plotpoint}}
\put(150.0,230.0){\usebox{\plotpoint}}
\put(151.0,230.0){\usebox{\plotpoint}}
\put(151.0,231.0){\usebox{\plotpoint}}
\put(152.0,231.0){\usebox{\plotpoint}}
\put(152.0,232.0){\usebox{\plotpoint}}
\put(153.0,232.0){\usebox{\plotpoint}}
\put(153.0,233.0){\usebox{\plotpoint}}
\put(154.0,233.0){\usebox{\plotpoint}}
\put(154.0,234.0){\usebox{\plotpoint}}
\put(155.0,234.0){\usebox{\plotpoint}}
\put(155.0,235.0){\usebox{\plotpoint}}
\put(156.0,235.0){\usebox{\plotpoint}}
\put(156.0,236.0){\usebox{\plotpoint}}
\put(157.0,236.0){\usebox{\plotpoint}}
\put(157.0,237.0){\rule[-0.200pt]{0.482pt}{0.400pt}}
\put(159.0,237.0){\usebox{\plotpoint}}
\put(159.0,238.0){\usebox{\plotpoint}}
\put(160.0,238.0){\usebox{\plotpoint}}
\put(160.0,239.0){\usebox{\plotpoint}}
\put(161.0,239.0){\usebox{\plotpoint}}
\put(161.0,240.0){\usebox{\plotpoint}}
\put(162.0,240.0){\usebox{\plotpoint}}
\put(162.0,241.0){\usebox{\plotpoint}}
\put(163.0,241.0){\usebox{\plotpoint}}
\put(163.0,242.0){\usebox{\plotpoint}}
\put(164.0,242.0){\usebox{\plotpoint}}
\put(164.0,243.0){\rule[-0.200pt]{0.482pt}{0.400pt}}
\put(166.0,243.0){\usebox{\plotpoint}}
\put(166.0,244.0){\usebox{\plotpoint}}
\put(167.0,244.0){\usebox{\plotpoint}}
\put(167.0,245.0){\usebox{\plotpoint}}
\put(168.0,245.0){\usebox{\plotpoint}}
\put(168.0,246.0){\usebox{\plotpoint}}
\put(169.0,246.0){\usebox{\plotpoint}}
\put(169.0,247.0){\rule[-0.200pt]{0.482pt}{0.400pt}}
\put(171.0,247.0){\usebox{\plotpoint}}
\put(171.0,248.0){\usebox{\plotpoint}}
\put(172.0,248.0){\usebox{\plotpoint}}
\put(172.0,249.0){\usebox{\plotpoint}}
\put(173.0,249.0){\usebox{\plotpoint}}
\put(173.0,250.0){\usebox{\plotpoint}}
\put(174.0,250.0){\usebox{\plotpoint}}
\put(174.0,251.0){\rule[-0.200pt]{0.482pt}{0.400pt}}
\put(176.0,251.0){\usebox{\plotpoint}}
\put(176.0,252.0){\usebox{\plotpoint}}
\put(177.0,252.0){\usebox{\plotpoint}}
\put(177.0,253.0){\usebox{\plotpoint}}
\put(178.0,253.0){\usebox{\plotpoint}}
\put(178.0,254.0){\rule[-0.200pt]{0.482pt}{0.400pt}}
\put(180.0,254.0){\usebox{\plotpoint}}
\put(180.0,255.0){\usebox{\plotpoint}}
\put(181.0,255.0){\usebox{\plotpoint}}
\put(181.0,256.0){\rule[-0.200pt]{0.482pt}{0.400pt}}
\put(183.0,256.0){\usebox{\plotpoint}}
\put(183.0,257.0){\usebox{\plotpoint}}
\put(184.0,257.0){\usebox{\plotpoint}}
\put(184.0,258.0){\usebox{\plotpoint}}
\put(185.0,258.0){\usebox{\plotpoint}}
\put(185.0,259.0){\rule[-0.200pt]{0.482pt}{0.400pt}}
\put(187.0,259.0){\usebox{\plotpoint}}
\put(187.0,260.0){\usebox{\plotpoint}}
\put(188.0,260.0){\usebox{\plotpoint}}
\put(188.0,261.0){\rule[-0.200pt]{0.482pt}{0.400pt}}
\put(190.0,261.0){\usebox{\plotpoint}}
\put(190.0,262.0){\usebox{\plotpoint}}
\put(191.0,262.0){\usebox{\plotpoint}}
\put(191.0,263.0){\rule[-0.200pt]{0.482pt}{0.400pt}}
\put(193.0,263.0){\usebox{\plotpoint}}
\put(193.0,264.0){\usebox{\plotpoint}}
\put(194.0,264.0){\usebox{\plotpoint}}
\put(194.0,265.0){\rule[-0.200pt]{0.482pt}{0.400pt}}
\put(196.0,265.0){\usebox{\plotpoint}}
\put(196.0,266.0){\usebox{\plotpoint}}
\put(197.0,266.0){\usebox{\plotpoint}}
\put(197.0,267.0){\rule[-0.200pt]{0.482pt}{0.400pt}}
\put(199.0,267.0){\usebox{\plotpoint}}
\put(199.0,268.0){\usebox{\plotpoint}}
\put(200.0,268.0){\usebox{\plotpoint}}
\put(200.0,269.0){\rule[-0.200pt]{0.482pt}{0.400pt}}
\put(202.0,269.0){\usebox{\plotpoint}}
\put(202.0,270.0){\usebox{\plotpoint}}
\put(203.0,270.0){\usebox{\plotpoint}}
\put(203.0,271.0){\rule[-0.200pt]{0.482pt}{0.400pt}}
\put(205.0,271.0){\usebox{\plotpoint}}
\put(205.0,272.0){\usebox{\plotpoint}}
\put(206.0,272.0){\usebox{\plotpoint}}
\put(206.0,273.0){\rule[-0.200pt]{0.482pt}{0.400pt}}
\put(208.0,273.0){\usebox{\plotpoint}}
\put(208.0,274.0){\rule[-0.200pt]{0.482pt}{0.400pt}}
\put(210.0,274.0){\usebox{\plotpoint}}
\put(210.0,275.0){\usebox{\plotpoint}}
\put(211.0,275.0){\usebox{\plotpoint}}
\put(211.0,276.0){\rule[-0.200pt]{0.482pt}{0.400pt}}
\put(213.0,276.0){\usebox{\plotpoint}}
\put(213.0,277.0){\rule[-0.200pt]{0.482pt}{0.400pt}}
\put(215.0,277.0){\usebox{\plotpoint}}
\put(215.0,278.0){\usebox{\plotpoint}}
\put(216.0,278.0){\usebox{\plotpoint}}
\put(216.0,279.0){\rule[-0.200pt]{0.482pt}{0.400pt}}
\put(218.0,279.0){\usebox{\plotpoint}}
\put(218.0,280.0){\rule[-0.200pt]{0.482pt}{0.400pt}}
\put(220.0,280.0){\usebox{\plotpoint}}
\put(220.0,281.0){\rule[-0.200pt]{0.482pt}{0.400pt}}
\put(222.0,281.0){\usebox{\plotpoint}}
\put(222.0,282.0){\usebox{\plotpoint}}
\put(223.0,282.0){\usebox{\plotpoint}}
\put(223.0,283.0){\rule[-0.200pt]{0.482pt}{0.400pt}}
\put(225.0,283.0){\usebox{\plotpoint}}
\put(225.0,284.0){\rule[-0.200pt]{0.482pt}{0.400pt}}
\put(227.0,284.0){\usebox{\plotpoint}}
\put(227.0,285.0){\rule[-0.200pt]{0.482pt}{0.400pt}}
\put(229.0,285.0){\usebox{\plotpoint}}
\put(229.0,286.0){\rule[-0.200pt]{0.482pt}{0.400pt}}
\put(231.0,286.0){\usebox{\plotpoint}}
\put(231.0,287.0){\rule[-0.200pt]{0.482pt}{0.400pt}}
\put(233.0,287.0){\usebox{\plotpoint}}
\put(233.0,288.0){\usebox{\plotpoint}}
\put(234.0,288.0){\usebox{\plotpoint}}
\put(234.0,289.0){\rule[-0.200pt]{0.482pt}{0.400pt}}
\put(236.0,289.0){\usebox{\plotpoint}}
\put(236.0,290.0){\rule[-0.200pt]{0.482pt}{0.400pt}}
\put(238.0,290.0){\usebox{\plotpoint}}
\put(238.0,291.0){\rule[-0.200pt]{0.482pt}{0.400pt}}
\put(240.0,291.0){\usebox{\plotpoint}}
\put(240.0,292.0){\rule[-0.200pt]{0.482pt}{0.400pt}}
\put(242.0,292.0){\usebox{\plotpoint}}
\put(242.0,293.0){\rule[-0.200pt]{0.482pt}{0.400pt}}
\put(244.0,293.0){\usebox{\plotpoint}}
\put(244.0,294.0){\rule[-0.200pt]{0.482pt}{0.400pt}}
\put(246.0,294.0){\usebox{\plotpoint}}
\put(246.0,295.0){\rule[-0.200pt]{0.482pt}{0.400pt}}
\put(248.0,295.0){\usebox{\plotpoint}}
\put(248.0,296.0){\rule[-0.200pt]{0.482pt}{0.400pt}}
\put(250.0,296.0){\usebox{\plotpoint}}
\put(250.0,297.0){\rule[-0.200pt]{0.482pt}{0.400pt}}
\put(252.0,297.0){\usebox{\plotpoint}}
\put(252.0,298.0){\rule[-0.200pt]{0.482pt}{0.400pt}}
\put(254.0,298.0){\usebox{\plotpoint}}
\put(254.0,299.0){\rule[-0.200pt]{0.482pt}{0.400pt}}
\put(256.0,299.0){\usebox{\plotpoint}}
\put(256.0,300.0){\rule[-0.200pt]{0.482pt}{0.400pt}}
\put(258.0,300.0){\usebox{\plotpoint}}
\put(258.0,301.0){\rule[-0.200pt]{0.482pt}{0.400pt}}
\put(260.0,301.0){\usebox{\plotpoint}}
\put(260.0,302.0){\rule[-0.200pt]{0.723pt}{0.400pt}}
\put(263.0,302.0){\usebox{\plotpoint}}
\put(263.0,303.0){\rule[-0.200pt]{0.482pt}{0.400pt}}
\put(265.0,303.0){\usebox{\plotpoint}}
\put(265.0,304.0){\rule[-0.200pt]{0.482pt}{0.400pt}}
\put(267.0,304.0){\usebox{\plotpoint}}
\put(267.0,305.0){\rule[-0.200pt]{0.482pt}{0.400pt}}
\put(269.0,305.0){\usebox{\plotpoint}}
\put(269.0,306.0){\rule[-0.200pt]{0.482pt}{0.400pt}}
\put(271.0,306.0){\usebox{\plotpoint}}
\put(271.0,307.0){\rule[-0.200pt]{0.723pt}{0.400pt}}
\put(274.0,307.0){\usebox{\plotpoint}}
\put(274.0,308.0){\rule[-0.200pt]{0.482pt}{0.400pt}}
\put(276.0,308.0){\usebox{\plotpoint}}
\put(276.0,309.0){\rule[-0.200pt]{0.482pt}{0.400pt}}
\put(278.0,309.0){\usebox{\plotpoint}}
\put(280,309.67){\rule{0.241pt}{0.400pt}}
\multiput(280.00,309.17)(0.500,1.000){2}{\rule{0.120pt}{0.400pt}}
\put(278.0,310.0){\rule[-0.200pt]{0.482pt}{0.400pt}}
\put(281,311){\usebox{\plotpoint}}
\put(281,311){\usebox{\plotpoint}}
\put(281,311){\usebox{\plotpoint}}
\put(281,311){\usebox{\plotpoint}}
\put(281,311){\usebox{\plotpoint}}
\put(281,311){\usebox{\plotpoint}}
\put(281,311){\usebox{\plotpoint}}
\put(281,311){\usebox{\plotpoint}}
\put(281,311){\usebox{\plotpoint}}
\put(281,311){\usebox{\plotpoint}}
\put(281,311){\usebox{\plotpoint}}
\put(281,311){\usebox{\plotpoint}}
\put(281,311){\usebox{\plotpoint}}
\put(281,311){\usebox{\plotpoint}}
\put(281,311){\usebox{\plotpoint}}
\put(281,311){\usebox{\plotpoint}}
\put(281,311){\usebox{\plotpoint}}
\put(281,311){\usebox{\plotpoint}}
\put(281,311){\usebox{\plotpoint}}
\put(281,311){\usebox{\plotpoint}}
\put(281,311){\usebox{\plotpoint}}
\put(281,311){\usebox{\plotpoint}}
\put(281,311){\usebox{\plotpoint}}
\put(281,311){\usebox{\plotpoint}}
\put(281,311){\usebox{\plotpoint}}
\put(281,311){\usebox{\plotpoint}}
\put(281,311){\usebox{\plotpoint}}
\put(281.0,311.0){\rule[-0.200pt]{0.482pt}{0.400pt}}
\put(283.0,311.0){\usebox{\plotpoint}}
\put(283.0,312.0){\rule[-0.200pt]{0.482pt}{0.400pt}}
\put(285.0,312.0){\usebox{\plotpoint}}
\put(285.0,313.0){\rule[-0.200pt]{0.723pt}{0.400pt}}
\put(288.0,313.0){\usebox{\plotpoint}}
\put(288.0,314.0){\rule[-0.200pt]{0.482pt}{0.400pt}}
\put(290.0,314.0){\usebox{\plotpoint}}
\put(290.0,315.0){\rule[-0.200pt]{0.482pt}{0.400pt}}
\put(292.0,315.0){\usebox{\plotpoint}}
\put(292.0,316.0){\rule[-0.200pt]{0.723pt}{0.400pt}}
\put(295.0,316.0){\usebox{\plotpoint}}
\put(295.0,317.0){\rule[-0.200pt]{0.482pt}{0.400pt}}
\put(297.0,317.0){\usebox{\plotpoint}}
\put(297.0,318.0){\rule[-0.200pt]{0.723pt}{0.400pt}}
\put(300.0,318.0){\usebox{\plotpoint}}
\put(300.0,319.0){\rule[-0.200pt]{0.482pt}{0.400pt}}
\put(302.0,319.0){\usebox{\plotpoint}}
\put(302.0,320.0){\rule[-0.200pt]{0.723pt}{0.400pt}}
\put(305.0,320.0){\usebox{\plotpoint}}
\put(305.0,321.0){\rule[-0.200pt]{0.482pt}{0.400pt}}
\put(307.0,321.0){\usebox{\plotpoint}}
\put(307.0,322.0){\rule[-0.200pt]{0.723pt}{0.400pt}}
\put(310.0,322.0){\usebox{\plotpoint}}
\put(310.0,323.0){\rule[-0.200pt]{0.723pt}{0.400pt}}
\put(313.0,323.0){\usebox{\plotpoint}}
\put(313.0,324.0){\rule[-0.200pt]{0.482pt}{0.400pt}}
\put(315.0,324.0){\usebox{\plotpoint}}
\put(315.0,325.0){\rule[-0.200pt]{0.723pt}{0.400pt}}
\put(318.0,325.0){\usebox{\plotpoint}}
\put(318.0,326.0){\rule[-0.200pt]{0.723pt}{0.400pt}}
\put(321.0,326.0){\usebox{\plotpoint}}
\put(321.0,327.0){\rule[-0.200pt]{0.482pt}{0.400pt}}
\put(323.0,327.0){\usebox{\plotpoint}}
\put(323.0,328.0){\rule[-0.200pt]{0.723pt}{0.400pt}}
\put(326.0,328.0){\usebox{\plotpoint}}
\put(326.0,329.0){\rule[-0.200pt]{0.723pt}{0.400pt}}
\put(329.0,329.0){\usebox{\plotpoint}}
\put(329.0,330.0){\rule[-0.200pt]{0.723pt}{0.400pt}}
\put(332.0,330.0){\usebox{\plotpoint}}
\put(332.0,331.0){\rule[-0.200pt]{0.723pt}{0.400pt}}
\put(335.0,331.0){\usebox{\plotpoint}}
\put(335.0,332.0){\rule[-0.200pt]{0.482pt}{0.400pt}}
\put(337.0,332.0){\usebox{\plotpoint}}
\put(337.0,333.0){\rule[-0.200pt]{0.723pt}{0.400pt}}
\put(340.0,333.0){\usebox{\plotpoint}}
\put(340.0,334.0){\rule[-0.200pt]{0.723pt}{0.400pt}}
\put(343.0,334.0){\usebox{\plotpoint}}
\put(343.0,335.0){\rule[-0.200pt]{0.723pt}{0.400pt}}
\put(346.0,335.0){\usebox{\plotpoint}}
\put(346.0,336.0){\rule[-0.200pt]{0.964pt}{0.400pt}}
\put(350.0,336.0){\usebox{\plotpoint}}
\put(350.0,337.0){\rule[-0.200pt]{0.723pt}{0.400pt}}
\put(353.0,337.0){\usebox{\plotpoint}}
\put(353.0,338.0){\rule[-0.200pt]{0.723pt}{0.400pt}}
\put(356.0,338.0){\usebox{\plotpoint}}
\put(356.0,339.0){\rule[-0.200pt]{0.723pt}{0.400pt}}
\put(359.0,339.0){\usebox{\plotpoint}}
\put(359.0,340.0){\rule[-0.200pt]{0.964pt}{0.400pt}}
\put(363.0,340.0){\usebox{\plotpoint}}
\put(363.0,341.0){\rule[-0.200pt]{0.723pt}{0.400pt}}
\put(366.0,341.0){\usebox{\plotpoint}}
\put(366.0,342.0){\rule[-0.200pt]{0.964pt}{0.400pt}}
\put(370.0,342.0){\usebox{\plotpoint}}
\put(370.0,343.0){\rule[-0.200pt]{0.964pt}{0.400pt}}
\put(374.0,343.0){\usebox{\plotpoint}}
\put(374.0,344.0){\rule[-0.200pt]{0.964pt}{0.400pt}}
\put(378.0,344.0){\usebox{\plotpoint}}
\put(378.0,345.0){\rule[-0.200pt]{1.204pt}{0.400pt}}
\put(383.0,345.0){\usebox{\plotpoint}}
\put(383.0,346.0){\rule[-0.200pt]{1.204pt}{0.400pt}}
\put(388.0,346.0){\usebox{\plotpoint}}
\put(388.0,347.0){\rule[-0.200pt]{1.445pt}{0.400pt}}
\put(394.0,347.0){\usebox{\plotpoint}}
\put(394.0,348.0){\rule[-0.200pt]{1.445pt}{0.400pt}}
\put(400.0,348.0){\usebox{\plotpoint}}
\put(400.0,349.0){\rule[-0.200pt]{1.927pt}{0.400pt}}
\put(408.0,349.0){\usebox{\plotpoint}}
\put(408.0,350.0){\rule[-0.200pt]{2.409pt}{0.400pt}}
\put(418.0,350.0){\usebox{\plotpoint}}
\put(418.0,351.0){\rule[-0.200pt]{3.854pt}{0.400pt}}
\put(434.0,351.0){\usebox{\plotpoint}}
\put(434.0,352.0){\rule[-0.200pt]{97.564pt}{0.400pt}}
\put(679,418){\makebox(0,0)[r]{$10x^2$}}
\multiput(699,418)(20.756,0.000){5}{\usebox{\plotpoint}}
\put(799,418){\usebox{\plotpoint}}
\put(110,304){\usebox{\plotpoint}}
\put(110.00,304.00){\usebox{\plotpoint}}
\put(118.76,316.00){\usebox{\plotpoint}}
\put(128.00,327.51){\usebox{\plotpoint}}
\put(138.00,338.27){\usebox{\plotpoint}}
\put(149.00,349.78){\usebox{\plotpoint}}
\put(160.00,359.53){\usebox{\plotpoint}}
\put(171.29,369.00){\usebox{\plotpoint}}
\put(183.63,378.00){\usebox{\plotpoint}}
\put(196.00,386.39){\usebox{\plotpoint}}
\put(208.14,395.00){\usebox{\plotpoint}}
\put(221.90,402.00){\usebox{\plotpoint}}
\put(234.65,410.00){\usebox{\plotpoint}}
\put(248.41,417.00){\usebox{\plotpoint}}
\put(262.16,424.00){\usebox{\plotpoint}}
\put(276.92,430.00){\usebox{\plotpoint}}
\put(291.68,436.00){\usebox{\plotpoint}}
\put(306.43,442.00){\usebox{\plotpoint}}
\put(321.19,448.00){\usebox{\plotpoint}}
\put(336.94,453.00){\usebox{\plotpoint}}
\put(352.70,458.00){\usebox{\plotpoint}}
\put(368.45,463.00){\usebox{\plotpoint}}
\put(385.21,467.00){\usebox{\plotpoint}}
\put(402.96,470.00){\usebox{\plotpoint}}
\put(421.72,472.00){\usebox{\plotpoint}}
\put(441.48,473.00){\usebox{\plotpoint}}
\put(462.23,473.00){\usebox{\plotpoint}}
\put(482.99,473.00){\usebox{\plotpoint}}
\put(503.74,473.00){\usebox{\plotpoint}}
\put(524.50,473.00){\usebox{\plotpoint}}
\put(545.25,473.00){\usebox{\plotpoint}}
\put(566.01,473.00){\usebox{\plotpoint}}
\put(586.76,473.00){\usebox{\plotpoint}}
\put(607.52,473.00){\usebox{\plotpoint}}
\put(628.27,473.00){\usebox{\plotpoint}}
\put(649.03,473.00){\usebox{\plotpoint}}
\put(669.79,473.00){\usebox{\plotpoint}}
\put(690.54,473.00){\usebox{\plotpoint}}
\put(711.30,473.00){\usebox{\plotpoint}}
\put(732.05,473.00){\usebox{\plotpoint}}
\put(752.81,473.00){\usebox{\plotpoint}}
\put(773.56,473.00){\usebox{\plotpoint}}
\put(794.32,473.00){\usebox{\plotpoint}}
\put(815.07,473.00){\usebox{\plotpoint}}
\put(835.83,473.00){\usebox{\plotpoint}}
\put(839,473){\usebox{\plotpoint}}
\put(110.0,131.0){\rule[-0.200pt]{0.400pt}{88.651pt}}
\put(110.0,131.0){\rule[-0.200pt]{175.616pt}{0.400pt}}
\put(839.0,131.0){\rule[-0.200pt]{0.400pt}{88.651pt}}
\put(110.0,499.0){\rule[-0.200pt]{175.616pt}{0.400pt}}
\end{picture}
\end{center}
\caption{The logarithmic progress rate $\ln f_t$  of RSH-I for maximising $x^2$ and $10x^2$.}
\label{f1-f3}
\end{figure}
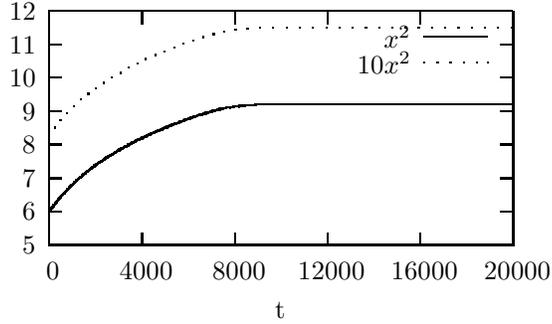
\end{example}

\section{Expected Hitting Times}

\label{secHittingTimes}

\subsection{Theoretical Study 1:: Fundamental Matrix}
In this subsection, we define the expected hitting time  and  fundamental matrix. The first hitting time is the number of iterations to find an optimal solution for the first time, which   is  an important measure of the performance of randomised search heuristics.  Its formal definition  is given as below. 

\begin{definition}
Suppose the initial population $\Phi_0 =X$. The mean   number   of iterations  when a randomised search heuristic encounters an optimal solution for the first time   is called the \emph{expected   hitting time}, denoted by $h(X)$.
\end{definition} 

When we talk about the expected hitting time,  we always assume that  randomised search heuristics  are convergent. Otherwise the expected hitting time is infinite, and that is out of our interest.
 
Let   $( X_1, X_2, \cdots   )$ represent all populations in the non-optimal set. Then the vector   
$$\mathbf{h}  = (h(X_1), h(X_2), \cdots   )^T,$$
represents   expected hitting times corresponding to all non-optimal populations. 

In absorbing Markov chains, the fundamental matrix plays a crucial role which is defined as follows.

\begin{definition}
\citep[Definition 11.3]{grinstead1997introduction} For an absorbing Markov chain $\{\Phi_t; t=0,1, \cdots \}$ with the transition matrix $\mathbf{Q}$, the matrix $\mathbf{N} = (\mathbf{I}-\mathbf{Q})^{-1}$ is
called the \emph{fundamental matrix}. 
\end{definition}
 
Now we explain the meaning of the entry $N(X,Y)$ of the fundamental matrix. Since the chain is convergent (i.e., $\rho(\mathbf{Q})<1$), then
\begin{align*}
\mathbf{N} = (\mathbf{I}-\mathbf{Q})^{-1}=\sum^{+\infty}_{t=0} \mathbf{Q}^t. 
\end{align*}
Rewriting the above equality in the entry form, we get
 \begin{align*}
N(X,Y) =  \sum^{+\infty}_{t=0} P(\Phi_t=Y \mid \Phi_0=X). 
\end{align*}
Therefore $N(X,Y)$  is the mean number of the chain visiting $Y$ when starting at $X$ \citep[Definition 11.3]{grinstead1997introduction}. 
  
The  hitting time vector $\mathbf{h}$ can be calculated by the fundamental matrix. According to Theorem 11.5 in
\citep{grinstead1997introduction}, for an absorbing Markov chain $\{ \Phi_t; t=0,1,  \cdots \}$,    its expected hitting times equal  to
\begin{align}
\mathbf{h}=\mathbf{N}\mathbf{1},
\end{align}
where $\mathbf{1}$ is a column vector all of whose entries are $1$.

However, it is difficult to apply the  above result to the analysis of expected hitting time since it is impossible to calculate the fundamental matrix in most cases.

\subsection{Theoretical Study 2: Average Drift Analysis}
In this subsection, we present average  drift analysis for bounding the expected hitting time, which is seldom investigated before. The first work to use less point-wise drift was \cite{jagerskupper2011combining}. Recently average drift analysis was applied to the runtime analysis of an EA for  unimodal functions.\footnote{Jun He, Tianshi Chen, Xin Yao: Average Drift Analysis and its Application. CoRR abs/1308.3080 (2013)}

It is too difficult to  calculate the expected hitting time   
through the fundamental matrix. Instead it is   more realistic to obtain their lower and upper bounds on  expected hitting time. Drift analysis was  introduced in bounding the expected hitting time of randomised search heuristics~\cite{he1998study,he2001drift}.
 In drift analysis,   $d(X)$ is called a \emph{drift function} if  $d(X) \ge 0$ for any non-optimal state $X$ and $d(X)=0$ for any  optimal state $X$. Given a drift function $d(X)$,   drift  
  represents the progress rate  of  moving towards the optima per iteration. 
\begin{definition}  
  \emph{Drift}    at   point $X$  is  defined by
$$
\Delta (X):=d(X)-\sum_{Y \in S_{\non}}  d(Y)  P(X,Y).
$$
\end{definition}

Let  $(X_1, X_2, \cdots)$   represent  all populations in the non-optimal set and  the vector 
$$\mathbf{d}=(d(X_1), d(X_2), \cdots)^T$$ represents  the drift function values corresponding to each non-optimal state. The
vector 
$$\boldsymbol{\Delta}=(\Delta(X_1), \Delta(X_2), \cdots)^T$$ represents  the drift  value corresponding to each non-optimal state.

The   drift $\Delta (X)$ is  determined by a single state $X$. So  it is called \emph{point-wise drift}.
Now we introduce   average drift which is the average of $\Delta (X)$ over the probability distribution of $\Phi_t=X$.

\begin{definition}
 \emph{Average   drift}   at the $t$-iteration is  
\begin{align}
\bar{\Delta}_t:=
   \sum_{X \in S_{\non}}\Delta(X)   \frac{P(\Phi_t =X)}{ P(\Phi_t \in S_{\non})}.
\end{align}
\end{definition}

 Let  $h(\Phi_0)$  denote the   expected hitting time  when  the initial population is $\Phi_0$, that is, 
$$h(\Phi_0)=\sum_{X \in S_{\non}}  h(X) P(\Phi_0=X).$$

Let $d(\Phi_0)$  denote the   expected drift function  when  the initial population is $\Phi_0$, that is,
$$d(\Phi_0)=\sum_{X \in S_{\non}}  d(X) P(\Phi_0=X).$$

The following average drift theorem is for upper-bounding  the expected hitting time.
\begin{theorem} 
\label{theDriftUpperBound}
Suppose a randomised search heuristic is convergent. If for any  $t   \ge 0$, the average    drift $
\bar{\Delta}_t \ge 1,
$ 
then the    expected hitting time $h(\Phi_0)$ is upper-bounded by
$d(\Phi_0).$
 \end{theorem}

\begin{proof}
Recall that the  $1$-norm       equals to
$$
\parallel \mathbf{q}_t \parallel_1= P(\Phi_t  \in S_{\non} ),
$$
then  the average    drift can be rewritten in an equivalent vector form:
 \begin{align}
\bar{\Delta}_t&=\frac{\mathbf{q}^T_t}{\parallel\mathbf{q}^T_t \parallel_1} (\mathbf{I}-\mathbf{Q})   \mathbf{d}.
\end{align}

The condition that  $
\bar{\Delta}_t \ge 1 
$  can be rewritten in an equivalent vector  form, 
 \begin{align*}
 \frac{\mathbf{q}^T_t}{\parallel\mathbf{q}^T_t \parallel_1} (\mathbf{I}-\mathbf{Q})   \mathbf{d} \ge 1.
\end{align*}

It follows 
 $$
\mathbf{q}^T_t (\mathbf{I}-\mathbf{Q})   \mathbf{d} \ge \parallel\mathbf{q}^T_t \parallel_1,
$$
then from   $\parallel\mathbf{q}^T_t \parallel_1=\mathbf{q}^T_t \mathbf{ 1}$, it follows
\begin{align}
\mathbf{q}^T_t (\mathbf{I}-\mathbf{Q})  \mathbf{d} \ge  \mathbf{q}^T_t \mathbf{ 1}.
\label{equForwardDrift3}
\end{align}

From the matrix iteration
$ 
\mathbf{q}^T_t = \mathbf{q}^T_0 \mathbf{Q}^t,
$ 
 it follows
$$
\mathbf{q}^T_0 \mathbf{Q}^t (\mathbf{I}-\mathbf{Q})   \mathbf{d} \ge  \mathbf{q}^T_0  \mathbf{Q}^t \mathbf{ 1},
$$

Equivalently
$$
\mathbf{q}^T_0 \left( \mathbf{Q}^t  -\mathbf{Q}^{t+1} \right) \mathbf{d} \ge  \mathbf{q}^T_0  \mathbf{Q}^t \mathbf{ 1},
$$

Now summing $t$ from $0$ to $k$, we get
$$
\sum^k_{t=0} \mathbf{q}^T_0 \left( \mathbf{Q}^t  -\mathbf{Q}^{t+1} \right) \mathbf{d} \ge \sum^k_{t=0} \mathbf{q}^T_0  \mathbf{Q}^t \mathbf{ 1},
$$
and simplifying both sides, it follows
\begin{align}
\label{equSimple}
\mathbf{q}^T_0    \mathbf{d}  -  \mathbf{q}^T_0 \mathbf{Q}^{k+1}   \mathbf{d} \ge  \mathbf{q}^T_0  \left ( \sum^k_{t=0} \mathbf{Q}^t \right) \mathbf{ 1}.
\end{align}

Due to $\rho(\mathbf{Q})<1$, the following two limits exist, 
\begin{align*}
&\lim_{k \to +\infty}\mathbf{Q}^{k+1}=\mathbf{O},\\
&\lim_{k \to +\infty} \sum^k_{t=0}   \mathbf{Q}^t =(\mathbf{I}-\mathbf{Q})^{-1}.
\end{align*}

Thus when $k \to +\infty$, (\ref{equSimple}) becomes
$$
  \mathbf{q}^T_0   \mathbf{d}     \ge  \mathbf{q}^T_0  (\mathbf{I}-\mathbf{Q})^{-1} \mathbf{ 1}.
$$

Recalling that
$ 
\mathbf{h} = \mathbf{N} \mathbf{1}=(\mathbf{I}-\mathbf{Q})^{-1} \mathbf{ 1},
$ 
we have
$$
  \mathbf{q}^T_0   \mathbf{d}     \ge  \mathbf{q}^T_0 \mathbf{h}.
$$ 
which proves the conclusion.
\end{proof}

Similarly we can establish an average drift theorem for  lower-bounding the expected hitting time. Its proof is the almost the same as that for the above theorem. We omit the   proof of theorem.
\begin{theorem} \label{theDriftLowerBound} 
If for any iteration $t   \ge 0$, the average  drift  $
\bar{\Delta}_t \le 1,
$ 
then the expected hitting time $h(\Phi_0)$ is lower-bounded by $d(\Phi_0)$.
\end{theorem}

  Previous point-wise drift   theorems~\citep[Theorems 2 and 3]{he2003towards}   are direct corollaries of current average drift theorems.

\begin{corollary} 
If for any non-optimal population $X$, its    drift  $  \Delta (X) \ge 1$, then the expected hitting time   $ h(X) \le d(X)$.
\end{corollary}

\begin{corollary}
 If for any non-optimal population $X$, its     drift  $  \Delta (X) \le 1$, then the expected hitting time  $ h(X) \ge d(X)$.
\end{corollary}
 
In point-wise drift theorems, its requirement is that the  drift is not less than 1 (or not more than 1) for all non-optimal states.  In average drift  theorems,  the condition is  replaced by that  average drift is not less than 1 (or not more than 1). Hence average drift analysis is more powerful than point-wise drift analysis.

\subsection{Theoretical Study 3: Backward Drift Analysis}
In this subsection, we present  novel backward drift analysis, which was never discussed before. We call the drift defined in the previous subsection  forward drift  in order to distinguish it from the backward drift introduced in the current subsection.
Forward and backward drift analysis can be regarded as a dual pair.

Starting from the fundamental matrix, we already know that the vector
$
\mathbf{h}  =\mathbf{N} \mathbf{1} 
$ represents  all expected hitting times in the non-optimal set.  
Similarly the vector
$$\mathbf{s}^T:=\mathbf{1}^T \mathbf{N}$$ gives another type of important times   for  randomised search heuristics. Now we explain the intuitive meaning of the vector $\mathbf{s}$. Notice that  the entry
\begin{align*}
s(Y)=\sum_{X \in S_{\non}} N(X,Y).
\end{align*}
and recall that $N(X,Y)$ is  the expected number of  that the Markov chain visits $Y$ when starting at  $X$, then $s(Y)$ is the sum of the expected number of visiting state $Y$ when starting from all non-optimal states.  We call $s(Y)$ the \emph{expected staying time} in non-optimal state $Y$.

The expected hitting time and expected staying time  have the following relationship. 

 \begin{theorem} 
 \label{theForwardBackward}
 Let $h(X)$ be the expected hitting time from non-optimal population  $X$ and $s(Y)$ the expected staying time in non-optimal population $Y$. Then 
\begin{align*}
 \sum_{X \in S_{\non}} h(X) = \sum_{Y \in S_{\non}} s(Y)
\end{align*}
\end{theorem}
 \begin{proof}
From the equalities
\begin{align*}
  \sum_{X \in S_{\non}} h(X)=  \mathbf{1}^T \mathbf{N}\mathbf{1}, \\
   \sum_{X \in S_{\non}} s(X)=   \mathbf{1}^T \mathbf{N}\mathbf{1},
\end{align*}
we draw the conclusion.
\end{proof} 

The above theorem  implies that the expected hitting time equals to the expected staying time when the initial population is chosen at uniformly random.  

Next we establish   backward drift analysis  for bounding the expected staying time.
Like forward drift analysis, a drift function is used in backward drift analysis. $d(Y)$ is  called a \emph{drift function} if $d(Y) \ge 0$ for any non-optimal state $Y$ and $d(Y)=0$ for any  optimal state $Y$.
 
\begin{definition}
Let $d(Y)$ be a drift function. For a  non-optimal population $Y$, the  backward forward drift   is  
$$
\nabla (Y):= d(Y)-\sum_{X \in S_{\non}}   d(X) P(X,Y).
$$
\end{definition} 

Backward drift is to measure the move from $Y$ to $X$ (backward). This is different from
forward drift
$$
\Delta (X):=d(X)-\sum_{Y \in S_{\non}}  d(Y)  P(X,Y).
$$
which  is to measure the move from $X$ to $Y$ (forward).

  The following backward drift theorem is used for upper-bounding the staying time.

\begin{theorem}
\label{theBackwardUpperBound}
 If for any non-optimal population $Y$, its    backward drift  $  \nabla (Y) \ge 1$, then the staying time  $ s(Y) \le d(Y)$.
\end{theorem}

\begin{proof}
Let  $(Y_1, Y_2, \cdots)$  represents   all non-optimal populations and  the vector 
$${\boldsymbol\nabla}^T=(\nabla (Y_1), \nabla (Y_2), \cdots)$$ 
represents their drift function values respectively.   Then from the definition,  $\boldsymbol{\nabla}^T$ can be rewritten in the vector form as follows:
 \begin{align}
    {\boldsymbol\nabla}^T:=\mathbf{d}^T (\mathbf{I}-\mathbf{Q}).
\end{align}

The condition that    $  \nabla (X) \ge 1$ is rewritten in a  vector form $
  \mathbf{d}^T (\mathbf{I}-\mathbf{Q}) \ge  \mathbf{1}^T  
$, and it follows
$
  \mathbf{d}^T (\mathbf{I}-\mathbf{Q})  - \mathbf{1}^T \ge \mathbf{0} ^T.
$

Since the fundamental matrix  $\mathbf{N} $ is non-negative, then
$$
  \left(\mathbf{d}^T (\mathbf{I}-\mathbf{Q})  - \mathbf{1}^T\right) \mathbf{N} \ge \mathbf{0}^T,
$$

Since $\mathbf{N} = (\mathbf{I}-\mathbf{Q})^{-1}$, so it follows
$
  \mathbf{d}^T   - \mathbf{1}^T  \mathbf{N} \ge \mathbf{0}^T,
$
and $\mathbf{s}^T \le \mathbf{d}^T$ which proves the conclusion.
\end{proof}

Similarly we can establish a backward drift theorem for lower-bounding the staying time.
\begin{theorem}
\label{theBackwardLowerBound}
 If for any non-optimal population $X$, its    backward drift  $ \nabla (X) \le 1$, then the staying time satisfies $ s(X) \ge d(X)$.
\end{theorem}

Backward drift analysis provides an alternative way of bounding the   expected hitting time  when the initial population is chosen at uniformly random. 
It is possible to establish average backward drift analysis similar to  average forward drift analysis. We will not discuss it in the detail.

\subsection{Case Studies}

In this subsection, we demonstrate how average drift analysis and backward drift analysis are applied to the estimation of the expected hitting time. 
There are three steps when applying drift analysis. First, choose an appropriate  drift function; then   estimate  drift; finally,  obtain  a bound  on the expected hitting time or staying time.

The first example is to show that average drift theorems work well but point-wise drift theorems are not applicable. 
\begin{example}
Consider RSH-I  for the maximizing problem    
$$\max \, x^2,   \quad x \in \{ 0,1, \cdots, 100 \}.$$ 
 
 Choose the drift function as  follows
\begin{equation*}
d(x)=\left\{
\begin{array}{lll}
\frac{100 \times 101}{99}(100-x), &\mbox{if } 1 \le x \le 99,\\
d(x+1), &\mbox{ if } x=0.  
\end{array}
\right.
\end{equation*}

Calculate the   drift $\Delta(x)$. For $1 \le x \le 99,$
\begin{align*}
\Delta(x)=&   d(x) -0.99 d(x)-0.01 d(x+1)=\frac{101}{99},\\
\Delta(0)=& d(0) -0.99 d(0)-0.01 d(1) =0.
\end{align*}
Notice that  $\Delta(0)= 0$, thus point-wise drift theorems  cannot be applied here.

However, average drift theorems  work well. Assume that the initial population $\Phi_0$ is chosen at uniformly  random, that is, $P(\Phi_0=x)=1/101$.    The average  forward drift  is  
\begin{align*}
  \bar{\Delta}_0 = \left(\frac{1}{101} \Delta(0) +  \frac{1}{101} \sum^{99}_{x=0}  \Delta(x)\right)  =1. 
\end{align*} 

Since RSH-I adopts elitist selection, so that if initial population $\Phi_0=x$ is not at state $0$, then for any $t \ge 0$, its offspring $\Phi_t$ never returns to state $0$. Thus  the average  forward drift  is not less than 
\begin{align*}
 \bar{\Delta}_t=  \left(\frac{1}{101} \Delta(0) +  \frac{1}{101} \sum^{99}_{x=0}  \Delta(x)\right)  =1. 
\end{align*}

According to Theorem~\ref{theDriftUpperBound},  the     expected hitting time  is not more than
$$
 \frac{1}{101} \sum^{99}_{x=0} d(x)= 5000.
$$
\end{example}

The second example is to show that we can derive the same conclusion using backward drift analysis.
 
\begin{example}
Still consider RSH-I  for the maximizing problem,   
$$\max \, x^2,   \quad x \in \{ 0,1, \cdots, 100 \}.$$

Choose the drift function as follow
$$d(x) =100(x+1) , \quad 0\le x \le 99,$$  

Calculate the  backward drift $\nabla(x)$. For $ 1 \le x \le 99,$
\begin{align*}
\nabla(x)=  &d(x)-0.99d(x)-0.01 d(x-1) =1, \\
\nabla(0)=  
  &d(0)-0.99d(0)=1.
\end{align*}

According to  Theorems~\ref{theBackwardUpperBound} and \ref{theBackwardLowerBound}, the staying time 
$$s(x)=d(x)=100(x+1).$$

Furthermore  if the initial population is chosen at uniformly random, then according to Theorem~\ref{theForwardBackward}, the   expected hitting time  equals to
\begin{align*}
\frac{1}{101} \sum^{99}_{x=0}s(x)=\frac{1}{101} \sum^{99}_{x=0} 100(x+1)=5000.
\end{align*}

The result is the same as that in the first example by applying average drift analysis.
\end{example}

\subsection{Computational Study}
In this subsection, we illustrate a computational approach to the calculation of the expected hitting time.

In the computational study,    we runs a randomised search heuristic for $k$ times. Let $\tau_i (X)$  be the  first hitting time for the $i$-th run during these $k$ runs.

  From the law of large numbers, the expected hitting time $h(x)$ is approximated by the average
\begin{align}
 \frac{1}{k}  \sum^k_{i=1} \tau_i (X), \mbox{ when } k \to +\infty.
\end{align}
The above average value is taken as the expected hitting time $h(X)$ in the computational study.

Sometimes it is easy to study the expected hitting time  through the computational study. 
\begin{example} 
Consider RSH-I and RSH-II for solving the maximization problem    
\begin{align*}
 \max \,  x^2, \qquad  {x \in \{ 0,1, \cdots, 100 \}}. 
\end{align*}   

We run each algorithm for 100,000 times. The initial population is $\Phi_0=20$.  Figure \ref{f1a-time} shows  that the expected hitting time of RSH-II is about $16000$, which  is twice as long as that of RSH-I, about $8000$.  

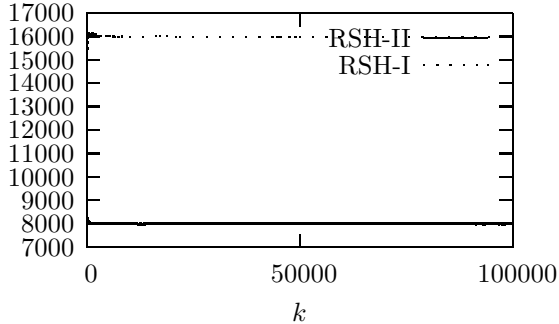
\begin{figure}[ht]
\begin{center}
\setlength{\unitlength}{0.240900pt}
\ifx\plotpoint\undefined\newsavebox{\plotpoint}\fi
\begin{picture}(900,540)(0,0)
\sbox{\plotpoint}{\rule[-0.200pt]{0.400pt}{0.400pt}}%
\put(170.0,131.0){\rule[-0.200pt]{4.818pt}{0.400pt}}
\put(150,131){\makebox(0,0)[r]{ 7000}}
\put(819.0,131.0){\rule[-0.200pt]{4.818pt}{0.400pt}}
\put(170.0,168.0){\rule[-0.200pt]{4.818pt}{0.400pt}}
\put(150,168){\makebox(0,0)[r]{ 8000}}
\put(819.0,168.0){\rule[-0.200pt]{4.818pt}{0.400pt}}
\put(170.0,205.0){\rule[-0.200pt]{4.818pt}{0.400pt}}
\put(150,205){\makebox(0,0)[r]{ 9000}}
\put(819.0,205.0){\rule[-0.200pt]{4.818pt}{0.400pt}}
\put(170.0,241.0){\rule[-0.200pt]{4.818pt}{0.400pt}}
\put(150,241){\makebox(0,0)[r]{ 10000}}
\put(819.0,241.0){\rule[-0.200pt]{4.818pt}{0.400pt}}
\put(170.0,278.0){\rule[-0.200pt]{4.818pt}{0.400pt}}
\put(150,278){\makebox(0,0)[r]{ 11000}}
\put(819.0,278.0){\rule[-0.200pt]{4.818pt}{0.400pt}}
\put(170.0,315.0){\rule[-0.200pt]{4.818pt}{0.400pt}}
\put(150,315){\makebox(0,0)[r]{ 12000}}
\put(819.0,315.0){\rule[-0.200pt]{4.818pt}{0.400pt}}
\put(170.0,352.0){\rule[-0.200pt]{4.818pt}{0.400pt}}
\put(150,352){\makebox(0,0)[r]{ 13000}}
\put(819.0,352.0){\rule[-0.200pt]{4.818pt}{0.400pt}}
\put(170.0,389.0){\rule[-0.200pt]{4.818pt}{0.400pt}}
\put(150,389){\makebox(0,0)[r]{ 14000}}
\put(819.0,389.0){\rule[-0.200pt]{4.818pt}{0.400pt}}
\put(170.0,425.0){\rule[-0.200pt]{4.818pt}{0.400pt}}
\put(150,425){\makebox(0,0)[r]{ 15000}}
\put(819.0,425.0){\rule[-0.200pt]{4.818pt}{0.400pt}}
\put(170.0,462.0){\rule[-0.200pt]{4.818pt}{0.400pt}}
\put(150,462){\makebox(0,0)[r]{ 16000}}
\put(819.0,462.0){\rule[-0.200pt]{4.818pt}{0.400pt}}
\put(170.0,499.0){\rule[-0.200pt]{4.818pt}{0.400pt}}
\put(150,499){\makebox(0,0)[r]{ 17000}}
\put(819.0,499.0){\rule[-0.200pt]{4.818pt}{0.400pt}}
\put(170.0,131.0){\rule[-0.200pt]{0.400pt}{4.818pt}}
\put(170,90){\makebox(0,0){ 0}}
\put(170.0,479.0){\rule[-0.200pt]{0.400pt}{4.818pt}}
\put(504.0,131.0){\rule[-0.200pt]{0.400pt}{4.818pt}}
\put(504,90){\makebox(0,0){ 50000}}
\put(504.0,479.0){\rule[-0.200pt]{0.400pt}{4.818pt}}
\put(839.0,131.0){\rule[-0.200pt]{0.400pt}{4.818pt}}
\put(839,90){\makebox(0,0){ 100000}}
\put(839.0,479.0){\rule[-0.200pt]{0.400pt}{4.818pt}}
\put(170.0,131.0){\rule[-0.200pt]{0.400pt}{88.651pt}}
\put(170.0,131.0){\rule[-0.200pt]{161.162pt}{0.400pt}}
\put(839.0,131.0){\rule[-0.200pt]{0.400pt}{88.651pt}}
\put(170.0,499.0){\rule[-0.200pt]{161.162pt}{0.400pt}}
\put(504,29){\makebox(0,0){$k$}}
\put(679,459){\makebox(0,0)[r]{RSH-II}}
\put(699.0,459.0){\rule[-0.200pt]{24.090pt}{0.400pt}}
\put(170,159){\usebox{\plotpoint}}
\put(170.0,159.0){\rule[-0.200pt]{0.400pt}{1.927pt}}
\put(170.0,140.0){\rule[-0.200pt]{0.400pt}{6.504pt}}
\put(170.0,140.0){\rule[-0.200pt]{0.400pt}{7.950pt}}
\put(170.0,169.0){\rule[-0.200pt]{0.400pt}{0.964pt}}
\put(170.0,169.0){\rule[-0.200pt]{0.400pt}{0.482pt}}
\put(170.0,168.0){\rule[-0.200pt]{0.400pt}{0.723pt}}
\put(170.0,168.0){\rule[-0.200pt]{0.400pt}{0.482pt}}
\put(170.0,169.0){\usebox{\plotpoint}}
\put(170.0,169.0){\rule[-0.200pt]{0.400pt}{1.927pt}}
\put(170.0,175.0){\rule[-0.200pt]{0.400pt}{0.482pt}}
\put(170.0,175.0){\usebox{\plotpoint}}
\put(170.0,174.0){\rule[-0.200pt]{0.400pt}{0.482pt}}
\put(170.0,174.0){\usebox{\plotpoint}}
\put(170.0,172.0){\rule[-0.200pt]{0.400pt}{0.723pt}}
\put(170.0,172.0){\rule[-0.200pt]{0.400pt}{1.445pt}}
\put(170.0,177.0){\usebox{\plotpoint}}
\put(170.0,177.0){\rule[-0.200pt]{0.400pt}{0.482pt}}
\put(170.0,175.0){\rule[-0.200pt]{0.400pt}{0.964pt}}
\put(170.0,175.0){\rule[-0.200pt]{0.400pt}{0.723pt}}
\put(170.0,176.0){\rule[-0.200pt]{0.400pt}{0.482pt}}
\put(170.0,176.0){\rule[-0.200pt]{0.400pt}{0.964pt}}
\put(170,176.67){\rule{0.241pt}{0.400pt}}
\multiput(170.00,177.17)(0.500,-1.000){2}{\rule{0.120pt}{0.400pt}}
\put(170.0,178.0){\rule[-0.200pt]{0.400pt}{0.482pt}}
\put(171,177){\usebox{\plotpoint}}
\put(171,177){\usebox{\plotpoint}}
\put(171.0,174.0){\rule[-0.200pt]{0.400pt}{0.723pt}}
\put(171.0,174.0){\usebox{\plotpoint}}
\put(171.0,174.0){\usebox{\plotpoint}}
\put(171.0,174.0){\rule[-0.200pt]{0.400pt}{0.482pt}}
\put(171.0,175.0){\usebox{\plotpoint}}
\put(171.0,175.0){\rule[-0.200pt]{0.400pt}{0.482pt}}
\put(171.0,175.0){\rule[-0.200pt]{0.400pt}{0.482pt}}
\put(171.0,175.0){\usebox{\plotpoint}}
\put(171.0,173.0){\rule[-0.200pt]{0.400pt}{0.723pt}}
\put(171.0,173.0){\usebox{\plotpoint}}
\put(171.0,173.0){\usebox{\plotpoint}}
\put(171.0,173.0){\usebox{\plotpoint}}
\put(171.0,172.0){\rule[-0.200pt]{0.400pt}{0.482pt}}
\put(171.0,172.0){\usebox{\plotpoint}}
\put(171.0,172.0){\usebox{\plotpoint}}
\put(171.0,172.0){\usebox{\plotpoint}}
\put(171.0,172.0){\usebox{\plotpoint}}
\put(171,171.67){\rule{0.241pt}{0.400pt}}
\multiput(171.00,172.17)(0.500,-1.000){2}{\rule{0.120pt}{0.400pt}}
\put(171.0,172.0){\usebox{\plotpoint}}
\put(172.0,172.0){\usebox{\plotpoint}}
\put(172.0,172.0){\usebox{\plotpoint}}
\put(172.0,172.0){\usebox{\plotpoint}}
\put(172.0,172.0){\usebox{\plotpoint}}
\put(172.0,172.0){\usebox{\plotpoint}}
\put(172.0,171.0){\rule[-0.200pt]{0.400pt}{0.482pt}}
\put(172.0,171.0){\rule[-0.200pt]{0.400pt}{0.482pt}}
\put(172.0,172.0){\usebox{\plotpoint}}
\put(172.0,172.0){\usebox{\plotpoint}}
\put(172.0,172.0){\usebox{\plotpoint}}
\put(172.0,172.0){\usebox{\plotpoint}}
\put(172.0,172.0){\usebox{\plotpoint}}
\put(172.0,172.0){\usebox{\plotpoint}}
\put(172.0,172.0){\usebox{\plotpoint}}
\put(172.0,172.0){\usebox{\plotpoint}}
\put(172.0,171.0){\rule[-0.200pt]{0.400pt}{0.482pt}}
\put(172.0,171.0){\usebox{\plotpoint}}
\put(172.0,171.0){\usebox{\plotpoint}}
\put(172.0,171.0){\rule[-0.200pt]{0.482pt}{0.400pt}}
\put(174.0,170.0){\usebox{\plotpoint}}
\put(174.0,170.0){\usebox{\plotpoint}}
\put(174.0,169.0){\rule[-0.200pt]{0.400pt}{0.482pt}}
\put(174.0,169.0){\usebox{\plotpoint}}
\put(175.0,169.0){\usebox{\plotpoint}}
\put(175.0,169.0){\usebox{\plotpoint}}
\put(175.0,169.0){\usebox{\plotpoint}}
\put(175.0,169.0){\usebox{\plotpoint}}
\put(175.0,169.0){\usebox{\plotpoint}}
\put(176.0,168.0){\usebox{\plotpoint}}
\put(176.0,168.0){\usebox{\plotpoint}}
\put(176.0,168.0){\usebox{\plotpoint}}
\put(176.0,168.0){\usebox{\plotpoint}}
\put(176.0,169.0){\rule[-0.200pt]{0.482pt}{0.400pt}}
\put(178.0,168.0){\usebox{\plotpoint}}
\put(178.0,168.0){\usebox{\plotpoint}}
\put(178.0,168.0){\usebox{\plotpoint}}
\put(178.0,168.0){\usebox{\plotpoint}}
\put(178.0,169.0){\usebox{\plotpoint}}
\put(179.0,168.0){\usebox{\plotpoint}}
\put(179.0,168.0){\usebox{\plotpoint}}
\put(179.0,168.0){\usebox{\plotpoint}}
\put(179.0,168.0){\rule[-0.200pt]{0.723pt}{0.400pt}}
\put(182.0,168.0){\usebox{\plotpoint}}
\put(182.0,168.0){\usebox{\plotpoint}}
\put(182.0,168.0){\usebox{\plotpoint}}
\put(182.0,168.0){\usebox{\plotpoint}}
\put(182.0,168.0){\rule[-0.200pt]{16.381pt}{0.400pt}}
\put(250.0,167.0){\usebox{\plotpoint}}
\put(250.0,167.0){\usebox{\plotpoint}}
\put(251.0,167.0){\usebox{\plotpoint}}
\put(251.0,167.0){\usebox{\plotpoint}}
\put(251.0,167.0){\usebox{\plotpoint}}
\put(251.0,167.0){\usebox{\plotpoint}}
\put(251.0,167.0){\usebox{\plotpoint}}
\put(251.0,167.0){\usebox{\plotpoint}}
\put(251.0,167.0){\usebox{\plotpoint}}
\put(251.0,168.0){\usebox{\plotpoint}}
\put(252.0,167.0){\usebox{\plotpoint}}
\put(252.0,167.0){\usebox{\plotpoint}}
\put(252.0,167.0){\usebox{\plotpoint}}
\put(252.0,167.0){\usebox{\plotpoint}}
\put(252.0,167.0){\usebox{\plotpoint}}
\put(252.0,167.0){\usebox{\plotpoint}}
\put(252.0,167.0){\usebox{\plotpoint}}
\put(252.0,167.0){\usebox{\plotpoint}}
\put(253.0,167.0){\usebox{\plotpoint}}
\put(253.0,168.0){\rule[-0.200pt]{0.723pt}{0.400pt}}
\put(256.0,167.0){\usebox{\plotpoint}}
\put(256.0,167.0){\usebox{\plotpoint}}
\put(256.0,167.0){\usebox{\plotpoint}}
\put(256.0,167.0){\usebox{\plotpoint}}
\put(256.0,167.0){\usebox{\plotpoint}}
\put(256.0,167.0){\usebox{\plotpoint}}
\put(256.0,168.0){\usebox{\plotpoint}}
\put(257.0,167.0){\usebox{\plotpoint}}
\put(257.0,167.0){\usebox{\plotpoint}}
\put(257.0,167.0){\usebox{\plotpoint}}
\put(257.0,167.0){\usebox{\plotpoint}}
\put(257.0,167.0){\usebox{\plotpoint}}
\put(257.0,167.0){\usebox{\plotpoint}}
\put(257.0,167.0){\usebox{\plotpoint}}
\put(257.0,167.0){\usebox{\plotpoint}}
\put(257.0,167.0){\usebox{\plotpoint}}
\put(257.0,167.0){\usebox{\plotpoint}}
\put(257.0,167.0){\usebox{\plotpoint}}
\put(257.0,167.0){\usebox{\plotpoint}}
\put(258.0,167.0){\usebox{\plotpoint}}
\put(258.0,167.0){\usebox{\plotpoint}}
\put(258.0,167.0){\usebox{\plotpoint}}
\put(258.0,167.0){\usebox{\plotpoint}}
\put(258.0,167.0){\usebox{\plotpoint}}
\put(258.0,167.0){\usebox{\plotpoint}}
\put(258.0,167.0){\usebox{\plotpoint}}
\put(258.0,167.0){\usebox{\plotpoint}}
\put(258.0,167.0){\usebox{\plotpoint}}
\put(258.0,167.0){\usebox{\plotpoint}}
\put(258.0,167.0){\usebox{\plotpoint}}
\put(258.0,168.0){\usebox{\plotpoint}}
\put(259.0,167.0){\usebox{\plotpoint}}
\put(259.0,167.0){\usebox{\plotpoint}}
\put(259.0,167.0){\usebox{\plotpoint}}
\put(259.0,167.0){\usebox{\plotpoint}}
\put(259.0,167.0){\usebox{\plotpoint}}
\put(259.0,167.0){\usebox{\plotpoint}}
\put(259.0,167.0){\usebox{\plotpoint}}
\put(259.0,167.0){\usebox{\plotpoint}}
\put(259.0,167.0){\usebox{\plotpoint}}
\put(259.0,167.0){\usebox{\plotpoint}}
\put(259.0,167.0){\usebox{\plotpoint}}
\put(259.0,167.0){\usebox{\plotpoint}}
\put(259.0,167.0){\usebox{\plotpoint}}
\put(259.0,167.0){\usebox{\plotpoint}}
\put(259.0,167.0){\usebox{\plotpoint}}
\put(259.0,167.0){\usebox{\plotpoint}}
\put(259,166.67){\rule{0.241pt}{0.400pt}}
\multiput(259.00,166.17)(0.500,1.000){2}{\rule{0.120pt}{0.400pt}}
\put(259.0,167.0){\usebox{\plotpoint}}
\put(260,168){\usebox{\plotpoint}}
\put(260,168){\usebox{\plotpoint}}
\put(260,168){\usebox{\plotpoint}}
\put(260,168){\usebox{\plotpoint}}
\put(260,168){\usebox{\plotpoint}}
\put(260.0,167.0){\usebox{\plotpoint}}
\put(260.0,167.0){\usebox{\plotpoint}}
\put(260.0,167.0){\usebox{\plotpoint}}
\put(260.0,167.0){\usebox{\plotpoint}}
\put(260.0,168.0){\rule[-0.200pt]{125.268pt}{0.400pt}}
\put(780.0,167.0){\usebox{\plotpoint}}
\put(780.0,167.0){\usebox{\plotpoint}}
\put(780.0,167.0){\usebox{\plotpoint}}
\put(780.0,167.0){\usebox{\plotpoint}}
\put(780.0,167.0){\usebox{\plotpoint}}
\put(780.0,167.0){\usebox{\plotpoint}}
\put(780.0,168.0){\rule[-0.200pt]{2.891pt}{0.400pt}}
\put(792.0,167.0){\usebox{\plotpoint}}
\put(792.0,167.0){\usebox{\plotpoint}}
\put(792.0,168.0){\usebox{\plotpoint}}
\put(793.0,167.0){\usebox{\plotpoint}}
\put(793.0,167.0){\usebox{\plotpoint}}
\put(793.0,167.0){\usebox{\plotpoint}}
\put(793.0,167.0){\usebox{\plotpoint}}
\put(793.0,167.0){\usebox{\plotpoint}}
\put(793.0,167.0){\usebox{\plotpoint}}
\put(793.0,167.0){\usebox{\plotpoint}}
\put(793.0,167.0){\usebox{\plotpoint}}
\put(794.0,167.0){\usebox{\plotpoint}}
\put(794.0,168.0){\rule[-0.200pt]{5.541pt}{0.400pt}}
\put(817.0,167.0){\usebox{\plotpoint}}
\put(817.0,167.0){\usebox{\plotpoint}}
\put(817.0,167.0){\usebox{\plotpoint}}
\put(817.0,167.0){\usebox{\plotpoint}}
\put(817.0,167.0){\usebox{\plotpoint}}
\put(817.0,167.0){\rule[-0.200pt]{1.204pt}{0.400pt}}
\put(822.0,167.0){\usebox{\plotpoint}}
\put(822.0,167.0){\usebox{\plotpoint}}
\put(822.0,167.0){\usebox{\plotpoint}}
\put(822.0,168.0){\rule[-0.200pt]{0.482pt}{0.400pt}}
\put(824.0,167.0){\usebox{\plotpoint}}
\put(824.0,167.0){\usebox{\plotpoint}}
\put(824.0,167.0){\usebox{\plotpoint}}
\put(824.0,167.0){\usebox{\plotpoint}}
\put(824.0,167.0){\usebox{\plotpoint}}
\put(824.0,167.0){\usebox{\plotpoint}}
\put(824.0,167.0){\usebox{\plotpoint}}
\put(824.0,167.0){\usebox{\plotpoint}}
\put(824.0,167.0){\usebox{\plotpoint}}
\put(824.0,167.0){\usebox{\plotpoint}}
\put(824.0,167.0){\usebox{\plotpoint}}
\put(824.0,167.0){\usebox{\plotpoint}}
\put(825.0,167.0){\usebox{\plotpoint}}
\put(825.0,167.0){\usebox{\plotpoint}}
\put(825.0,167.0){\usebox{\plotpoint}}
\put(825.0,167.0){\usebox{\plotpoint}}
\put(825.0,167.0){\usebox{\plotpoint}}
\put(825.0,167.0){\usebox{\plotpoint}}
\put(825.0,167.0){\usebox{\plotpoint}}
\put(825.0,167.0){\usebox{\plotpoint}}
\put(825.0,167.0){\usebox{\plotpoint}}
\put(825.0,167.0){\usebox{\plotpoint}}
\put(825.0,167.0){\usebox{\plotpoint}}
\put(825.0,167.0){\usebox{\plotpoint}}
\put(825.0,167.0){\usebox{\plotpoint}}
\put(826.0,167.0){\usebox{\plotpoint}}
\put(826.0,167.0){\usebox{\plotpoint}}
\put(826.0,167.0){\usebox{\plotpoint}}
\put(826.0,167.0){\usebox{\plotpoint}}
\put(826.0,167.0){\usebox{\plotpoint}}
\put(826.0,168.0){\rule[-0.200pt]{3.132pt}{0.400pt}}
\put(679,418){\makebox(0,0)[r]{RSH-I}}
\multiput(699,418)(20.756,0.000){5}{\usebox{\plotpoint}}
\put(799,418){\usebox{\plotpoint}}
\put(170,458){\usebox{\plotpoint}}
\multiput(170,458)(0.000,-20.756){4}{\usebox{\plotpoint}}
\multiput(170,406)(0.000,20.756){2}{\usebox{\plotpoint}}
\put(170.00,431.47){\usebox{\plotpoint}}
\put(170.00,410.71){\usebox{\plotpoint}}
\put(170.00,392.04){\usebox{\plotpoint}}
\put(170.00,406.80){\usebox{\plotpoint}}
\put(170.00,427.55){\usebox{\plotpoint}}
\put(170.00,437.69){\usebox{\plotpoint}}
\put(170.00,429.07){\usebox{\plotpoint}}
\put(170.00,421.82){\usebox{\plotpoint}}
\put(170.00,424.58){\usebox{\plotpoint}}
\put(170.00,435.33){\usebox{\plotpoint}}
\put(170.00,440.09){\usebox{\plotpoint}}
\put(170.00,444.84){\usebox{\plotpoint}}
\put(170.00,444.40){\usebox{\plotpoint}}
\put(171.00,443.35){\usebox{\plotpoint}}
\put(171.00,452.11){\usebox{\plotpoint}}
\put(171.00,453.13){\usebox{\plotpoint}}
\put(171.00,460.38){\usebox{\plotpoint}}
\put(172.00,463.38){\usebox{\plotpoint}}
\put(172.00,463.87){\usebox{\plotpoint}}
\put(173.00,467.89){\usebox{\plotpoint}}
\put(173.00,465.36){\usebox{\plotpoint}}
\put(174.00,465.60){\usebox{\plotpoint}}
\put(175.00,466.15){\usebox{\plotpoint}}
\put(176.00,464.09){\usebox{\plotpoint}}
\put(178.00,466.66){\usebox{\plotpoint}}
\put(179.00,465.58){\usebox{\plotpoint}}
\put(180.00,464.18){\usebox{\plotpoint}}
\put(182.00,465.07){\usebox{\plotpoint}}
\put(184.00,465.69){\usebox{\plotpoint}}
\put(191.44,463.00){\usebox{\plotpoint}}
\put(200.20,463.00){\usebox{\plotpoint}}
\put(205.00,462.95){\usebox{\plotpoint}}
\put(206.00,462.71){\usebox{\plotpoint}}
\put(211.00,462.46){\usebox{\plotpoint}}
\put(214.22,462.00){\usebox{\plotpoint}}
\put(219.00,462.02){\usebox{\plotpoint}}
\put(224.00,462.27){\usebox{\plotpoint}}
\put(225.00,462.51){\usebox{\plotpoint}}
\put(237.24,462.00){\usebox{\plotpoint}}
\put(258.00,462.00){\usebox{\plotpoint}}
\put(278.00,462.75){\usebox{\plotpoint}}
\put(284.51,463.00){\usebox{\plotpoint}}
\put(295.26,463.00){\usebox{\plotpoint}}
\put(309.00,462.02){\usebox{\plotpoint}}
\put(315.77,462.00){\usebox{\plotpoint}}
\put(336.53,462.00){\usebox{\plotpoint}}
\put(357.29,462.00){\usebox{\plotpoint}}
\put(378.04,462.00){\usebox{\plotpoint}}
\put(398.80,462.00){\usebox{\plotpoint}}
\put(419.55,462.00){\usebox{\plotpoint}}
\put(440.31,462.00){\usebox{\plotpoint}}
\put(461.00,461.94){\usebox{\plotpoint}}
\put(468.00,461.82){\usebox{\plotpoint}}
\put(472.00,461.43){\usebox{\plotpoint}}
\put(473.33,462.00){\usebox{\plotpoint}}
\put(487.67,461.00){\usebox{\plotpoint}}
\put(496.00,461.43){\usebox{\plotpoint}}
\put(498.00,461.82){\usebox{\plotpoint}}
\put(503.94,461.00){\usebox{\plotpoint}}
\put(524.69,461.00){\usebox{\plotpoint}}
\put(545.45,461.00){\usebox{\plotpoint}}
\put(566.20,461.00){\usebox{\plotpoint}}
\put(586.96,461.00){\usebox{\plotpoint}}
\put(605.71,461.00){\usebox{\plotpoint}}
\put(613.47,462.00){\usebox{\plotpoint}}
\put(634.23,462.00){\usebox{\plotpoint}}
\put(654.98,462.00){\usebox{\plotpoint}}
\put(675.74,462.00){\usebox{\plotpoint}}
\put(696.49,462.00){\usebox{\plotpoint}}
\put(717.25,462.00){\usebox{\plotpoint}}
\put(738.00,462.00){\usebox{\plotpoint}}
\put(758.76,462.00){\usebox{\plotpoint}}
\put(779.51,462.00){\usebox{\plotpoint}}
\put(800.27,462.00){\usebox{\plotpoint}}
\put(821.03,462.00){\usebox{\plotpoint}}
\put(839,462){\usebox{\plotpoint}}
\put(170.0,131.0){\rule[-0.200pt]{0.400pt}{88.651pt}}
\put(170.0,131.0){\rule[-0.200pt]{161.162pt}{0.400pt}}
\put(839.0,131.0){\rule[-0.200pt]{0.400pt}{88.651pt}}
\put(170.0,499.0){\rule[-0.200pt]{161.162pt}{0.400pt}}
\end{picture}
\end{center} 
\caption{Expected hitting times of RSH-I and RSH-II for maximising $f_2(x)$. $k$ is the number of runs.}
\label{f1a-time}
\end{figure} 
\end{example}

However, sometimes the calculation of the expected hitting time needs an extremely long computation time.
 
\begin{example} 
Consider  RSH-II for solving the maximization problem    
\begin{align*}
 \max \,  (x-49)^2, \qquad  {x \in \{ 0,1, \cdots, 100 \}}. 
\end{align*}   

We run each algorithm for 100,000 times. We found that RSH-II always got stuck at the local optimum $0$. Even for the simple problem, it seems not easy to make a computational study of the expected hitting time.
\end{example}

\section{Conclusions and Future Work}
\label{secConclusion}
A unified Markov chain approach has been proposed for studying the convergence, convergence rate and   expected hitting time  of randomised search heuristics in this paper. The core of the approach is to  model  randomised search heuristics by   absorbing Markov chains and  then to study the chain based on matrix iteration analysis. A novelty in the analysis is that  the vector 1-norm is used to represent the probability of a population in the non-optimal solution set. It plays the role of the distance between a population and the optimal solution set. All theoretical results are proven in a unified manner. 

The results of the paper are summarised as follow:  First, Theoreom~\ref{theConvergence} establishes a  sufficient and necessary condition of convergence in distribution. The theorem states a randomised search heuristic is convergent if and only if the algorithm can make an improvement  in  finite iterations suppose  the current solution is not optimal. 

Then    the   average convergence rate is introduced, which  refers to the average reduction factor of the probability of a population in the non-optimal set per iteration in terms of  the   logarithmic mean.    Theorems~\ref{theRateLowerBound} and \ref{theRateUpperBound} provide   lower and upper bounds on the average convergence rate. 

Finally, two new types of drift analysis, average drift analysis and backward drift analysis,  are proposed for analysing the expected hitting time. Theorems~\ref{theDriftUpperBound} and \ref{theDriftLowerBound} state that the expected hitting time can be bounded by a drift function and related average drift.     
Theorem~\ref{theForwardBackward} reveals that the expected hitting time and expected staying time are equal if the initial population is chosen at uniformly random.  Theorems~\ref{theBackwardUpperBound} and \ref{theBackwardLowerBound} state that  the expected staying time can be bounded by a drift function and related backward drift.

Besides the theoretical study, computation approaches are also presented to study the convergence, average convergence rate and expected hitting time. 

Comparing computational  and theoretical studies, we see none of them are perfect. The computational study may provide an intuitive description of randomised search heuristics' behaviour, but it belongs to a posterior and case study.  
 On  the other hand, the theoretical study belongs to a prior and general study. Theoretical results may provide some understanding of randomised search heuristics' ability.   But it is still too hard  to obtain an exact value of the expected hitting time or average convergence rate. 
 
 There is a   gap between the theoretical and computational studies. Our future work is to improve  theoretical tools, to apply them to different types of randomised search heuristics and  to study their convergence, convergence rate and expected hitting time.  

\end{document}